\documentclass[11pt, letterpaper]{amsart}
\usepackage[left=1in,right=1in,bottom=1in,top=1in]{geometry}
\usepackage{amsmath,amsthm,amssymb}
\usepackage{calligra,mathrsfs}
\usepackage[all,cmtip]{xy}
\usepackage{cite}
\usepackage{hyperref}
\usepackage{enumerate}
\usepackage{comment}
\usepackage{graphicx}
\usepackage{enumitem} \setlist[enumerate]{label={\upshape(\roman*)}}

\usepackage[T2A,T1]{fontenc}
\usepackage[utf8]{inputenc}
\usepackage[russian,english]{babel}

\newcommand{\on}[1]{\operatorname{#1}}
\newcommand{\mathfont}{\mathbf}

\newcommand{\ZZ}{\mathfont Z}

\newcommand{\QQ}{\mathfont Q}

\newcommand{\FF}{\mathfont F}

\newcommand{\PP}{\mathfont{P}}

\newcommand{\pfrak}{\mathfrak{p}}
\newcommand{\p}{\pfrak}

\newcommand{\m} {\mathfrak m}
\newcommand{\n} {\mathfrak n}

\DeclareFontFamily{OT1}{rsfs}{}
\DeclareFontShape{OT1}{rsfs}{n}{it}{<-> rsfs10}{}
\DeclareMathAlphabet{\mathscr}{OT1}{rsfs}{n}{it}

\newcommand{\Gscr}{\mathscr{G}}

\newcommand{\T}{\mathscr{T}}
\newcommand{\Qscr}{\mathscr{Q}}
\newcommand{\wh}[1]{\widehat{#1}}

\newcommand{\Hom}{\on{Hom}}
\newcommand{\coker} {\on{coker} }

\newcommand*{\tens}{\mathop{\otimes}\limits}
\renewcommand{\Im}{\on{im}}
\renewcommand{\div}{\on{div}}
\renewcommand{\null}{\on{null}}

\newcommand{\Aut}{\on{Aut}}

\newcommand{\Gal}{\on{Gal}}

\newcommand{\ord}{\on{ord}}

\newcommand{\rank}{\on{rank}}

\mathchardef\mhyphen="2D
\newcommand{\FVnil}{F,V\mhyphen\mathrm{nil}}
\newcommand{\Fnil}{F\mhyphen\mathrm{nil}}
\newcommand{\Vnil}{V\mhyphen\mathrm{nil}}
\newcommand{\Vbij}{V\mhyphen\mathrm{bij}}
\newcommand{\Fbij}{F\mhyphen\mathrm{bij}}

\DeclareMathOperator{\id}{id}
\DeclareMathOperator{\dR}{dR}
\DeclareMathOperator{\Tr}{tr}
\DeclareMathOperator{\op}{op}

\DeclareMathOperator{\BT}{BT}
\DeclareMathOperator{\Tor}{Tor}
\DeclareMathOperator{\Ext}{Ext}

\DeclareMathOperator{\length}{length}

\DeclareMathOperator{\Ht}{ht}
\DeclareMathOperator{\Fr}{F}
\DeclareMathOperator{\Rad}{Rad}

\newcommand{\D}{\mathbf{D}}

\newcommand{\M}{M}
\renewcommand{\O}{\mathcal{O}}
\newcommand{\calD}{\mathcal{D}}

\newcommand{\res}{\on{res}}
\newcommand{\Res}{\on{Res}}
\newcommand{\Ind}{\on{Ind}}
\newcommand{\CoInd}{\on{CoInd}}

\newcommand{\Spec}{\on{Spec}}

\newcommand{\Lie}{\on{Lie}}

\newcommand{\Gm}{\mathfont{G}_m}

\newcommand{\pr}{\on{pr}}

\theoremstyle{plain}
\newtheorem{lem}{Lemma}
\newtheorem{thm}[lem]{Theorem}
\newtheorem{thmx}{Theorem}
\newtheorem{prop}[lem]{Proposition}
\newtheorem{cor}[lem]{Corollary}

\theoremstyle{definition}
\newtheorem{defn}[lem]{Definition}

\newtheorem{remark}[lem]{Remark}

\numberwithin{equation}{section}
\numberwithin{lem}{section}

\DeclareMathOperator{\HOM}{\mathscr{H}\text{\kern -3pt {\calligra\large om}}\,}

\newcommand*{\I}{\mathscr{I}}
\newcommand*{\J}{\mathscr{J}}
\renewcommand*{\L}{\mathscr{L}}

\renewcommand{\o}[1]{\overline{#1}}
\DeclareMathOperator{\Pic}{Pic}
\DeclareMathOperator{\Alb}{Alb}
\DeclareMathOperator{\Cl}{Cl}
\DeclareMathOperator{\Nm}{Nm}
\DeclareMathOperator{\red}{red}

\newcommand*{\et}{\mathrm{\acute{e}t}}
\newcommand*{\mult}{\mathrm{m}}
\renewcommand*{\ll}{\mathrm{ll}}

\title{Iwasawa Dieudonn\'e theory of function fields}
\author{Bryden Cais}
\date{\today}

\email{cais@math.arizona.edu}
\address{Department of Mathematics\\
  University of Arizona\\
  Tucson, Arizona 85721}

\usepackage[usenames,dvipsnames]{color}
\newcommand{\bryden}[1]{{\color{Red} \sf $\bigstar\bigstar$  [#1]}}

\begin{document} 

\CompileMatrices
\UseTips

\begin{abstract}
    Let $k$ be a perfect field of characteristic $p$ and $\Gamma$ an infinite, first countable
    pro-$p$ group.  We study the behavior of the $p$-primary part of the ``motivic class group'',
    {\em i.e.}~the full $p$-divisible group of the Jacobian, in any $\Gamma$-tower of function fields
    over $k$ that is unramified outside a finite (possibly empty) set of places $\Sigma$, and totally ramified
    at every place of $\Sigma$.
    When $\Sigma=\emptyset$ and $\Gamma$ is a torsion free $p$-adic Lie group, we obtain asymptotic formulae
    which show that the $p$-torsion class group schemes grow in a remarkably regular manner.
    In the ramified setting $\Sigma\neq\emptyset$, 
    we obtain a similar asymptotic formula for the $p$-torsion in ``physical class groups'', {\em i.e.}~the $k$-rational points of the Jacobian,
    which generalizes the work of Mazur and Wiles, who studied the case $\Gamma=\ZZ_p$.
 \end{abstract}
 

\maketitle

\section{Introduction}\label{intro}

Let $k$ be a perfect field of characteristic $p>0$ and $K$ an algebraic function field in one variable
over $k$.  Let $L/K$ be a Galois extension, unramified outside a finite
(possibly empty)
set $\Sigma$ of places\footnote{We require that these places are trivial on $k$, so that they correspond to closed points of the smooth projective algebraic curve over $k$ associated to $K$.  This is automatic if $k$ is finite.} of $K$, with $\Gamma:=\Aut(L/K)$ 
an infinite pro-$p$ group.
We assume that $k$ is algebraically closed in $L$, that every place in $\Sigma$ totally ramifies
in $L$, and that $\Gamma$ admits a countable basis $\Gamma = \Gamma_0\supsetneq \Gamma_1\supsetneq \cdots$ of the identity consisting of open normal
subgroups.\footnote{This condition is needed in order to prove that certain derived limits vanish,
and is automatic whenever $k$ itself is countable; see Remarks \ref{limvanish} and \ref{kcountable}}
Let $X_n$ be
the unique 
smooth, proper, and geometrically connected curve over $k$ whose function field
is $L^{\Gamma_n}$ (the fixed field of $\Gamma_n$), and for $n \ge m$
let $\pi_{n,m}: X_n\rightarrow X_m$ be the corresponding branched
Galois covering of curves.  
In this way, we obtain a $\Gamma$-{\em tower}
of curves and our aim is to understand the growth of the $p$-primary part of the {\em motivic class group}
as we move up this tower.

More precisely, let $J_n:=\Pic^0_{X_n/k}$ be the Jacobian of $X_n$ over $k$,
and for $n \ge m$ let $\pi_{n,m}^*: J_m\rightarrow J_n$
be the map on Jacobians induced from $\pi_{m,n}$ by Picard functoriality ({\em i.e.} pullback of line bundles).
The main object of study in this paper is the inductive system of $p$-divisible groups
\begin{equation}
	\Gscr_n:=J_n[p^{\infty}] 
\end{equation}
with transition maps $\pi_{n,m}^*:\Gscr_m\rightarrow \Gscr_n$ for every $n \ge m$. 
We view each $\Gscr_n$ as the $p$-primary part of the {\em motivic class group} $J_n$
of $X_n$ (as opposed to the usual ``physical'' class group $\Cl_{X_n}:=J_n(k)$ of $k$-rational 
points on $J_n$), and in the spirit of Iwasawa theory, our aim is to understand the growth---broadly construed---of $\Gscr_n$ as $n\rightarrow \infty$.

To do this, we will linearize the inherently geometric objects $\Gscr_n$ by passing
to their (contravariant) Dieudonn\'e modules, or equivalently \cite{MM} we will study the first crystalline
cohomology groups of the curves $X_n$; by Dieudonn\'e theory, this passage incurs no loss of information,
as any $p$-divisible group may be functorially recovered from its Dieudonn\'e module.
Since $k$ is perfect, for each $n$ there is a functorial
decomposition of $p$-divisible groups 
\begin{equation*}
	\Gscr_n = \Gscr_n^{\et} \times_k \Gscr_n^{\mult} \times_k \Gscr_n^{\ll}
\end{equation*}
into \'etale (reduced with local dual), multiplicative (local with reduced dual), and local-local components.
For $\star\in \{\et,\mult,\ll\}$, we introduce the {\em Iwasawa--Dieudonn\'e module}
\begin{equation}\label{ID}
	\D^{\star}:= \varprojlim_{n} \D(\Gscr_{n}^{\star}),
\end{equation}
with the projective limit taken via the functorially induced transition maps.
This is naturally a (left) topological module over the 
completed group algebra
\begin{equation*}
	\Lambda:=
	\varprojlim_{n} \Lambda_n\quad\text{for}\quad \Lambda_n:=W[\Gamma/\Gamma_n] 
\end{equation*}
with $W=W(k)$ the ring of Witt vectors of $k$.
The Frobenius automorphism $\sigma$ of $W$
induces an automorphism of $\Lambda$ that acts trivially on $\Gamma$,
which we again denote by $\sigma$.
Each $\D^{\star}$ 
comes equipped with continuous, additive maps $F,V:\D^{\star}\rightarrow \D^{\star}$
(Frobenius and Verscheibung)
that are semilinear over $\sigma$ and $\sigma^{-1}$, respectively, and satisfy $FV=VF=p$.
In this paper, we will analyze the $\Lambda$-module structure of these Iwasawa--Dieudonn\'e modules,
proving in all but one case that they are finitely generated and very nearly free modules
satisfying {\em control} in the sense that each finite level 
$\D(\Gscr_n^{\star})$ may be recovered from the corresponding Iwasawa--Dieudonn\'e module.

Our work generalizes that of Mazur and Wiles \cite{mw83}
and Crew \cite{crew84}
who in effect\footnote{Neither Mazur and Wiles nor Crew explicitly work with Dieudonn\'e modules,
instead phrasing their work in terms of contravariant Tate modules, which capture only  the \'etale parts of the $p$-divisible groups.} studied 
$\D^{\et}$ in the case of ramified ({\em i.e.}~$\Sigma\neq\emptyset$) $\Gamma=\ZZ_p$-towers
and used their work to prove that, when $k$ is finite, the $\ZZ_p[\![\Gamma]\!]$-module $\varprojlim_n \Hom_{\ZZ_p}(\Cl_{X_n}[p^{\infty}],\QQ_p/\ZZ_p)$
of Pontryagin duals of $p$-primary components of ``physical'' class groups is finitely generated and torsion;
Iwasawa's celebrated formula
$|\Cl_{X_n}[p^{\infty}]|= p^{\mu p^n+\lambda n + \nu}$ (for $n\gg 0$) follows.
For {\em arbitrary} $\Gamma$, we obtain analogous structure and control theorems for $\D^{\et}$
and $\D^{\mult}$ in the ramified case, and for all three of $\D^{\et},\D^{\mult},$ and $\D^{\ll}$
in the case $\Sigma=\emptyset$ of everywhere unramified $\Gamma$-towers.
In the particular case of the Igusa tower (wherein $\Gamma=\ZZ_p$), the $\Lambda$-modules $\D^{\et}$ and $\D^{\mult}$ were first studied in \cite{CaisHida2}, and play a central role in the author's work on geometric Hida theory.  \nocite{CaisHida1}

In a different---but intimately connected---direction, Wan 
has initiated the study of Zeta functions in $\Gamma$-towers of curves over finite fields.
In \cite{WanSlopes2} and \cite{WanSlopes1}, 
he and his collaborators prove
that for certain kinds of ramified $\Gamma=\ZZ_p^d$-towers with $X_0=\PP^1$ and $\Sigma=\{\infty\}$,
the Newton slopes of the zeta functions of the curves $X_n$ behave in a remarkably regular manner
as $n\rightarrow \infty$; see also \cite{LiSlopes} and \cite{RenSlopes}. These slopes are also the slopes of the $F$-isocrystal $\D(\Gscr_n)[\frac{1}{p}]$, so by Dieudon\'e theory and the Dieudonn\'e--Manin classification of isocrystals, they determine $\Gscr_n$ up to isogeny.
In this way, Wan's program can be interpreted as providing a beautiful 
analogue of classical Iwasawa theory for the {\em isogeny types} of the $p$-divisible groups $\Gscr_n$
in these $\Gamma=\ZZ_p^d$-towers.  However,    
{\em isogeny type} loses touch with torsion phenomena, which the Iwasawa--Dieudonn\'e modules
do capture.  In addition to allowing arbitrary $\Gamma$ as in the introduction, 
the results of this paper also apply in the case of arbitrary perfect $k$,
whereas any study of $L$-functions in $\Gamma$-towers of curves would seem to be limited to
{\em finite} base fields.  In particular, we are able to analyze {\em \'etale} $\Gamma$-towers
(see Remark Remark \ref{unrkhyp}), and obtain rather complete results in this case. 
Using these results, 
when $\Gamma$ is a torsion-free $p$-adic Lie group
we obtain an asymptotic formula (Corollary \ref{thmI}) 
which shows that
the local--local $p$-torsion {\em class group schemes} $\Gscr_n^{\ll}[p]$
grow in a remarkably regular manner, and provides a fundamentally new kind of
Iwasawa-theoretic result for unramified $p$-adic Lie extensions.
We also give a new proof 
of Shafarevich's theorem \cite[\foreignlanguage{russian}{Tеорема }2]{Shafarevich}
 ({\em cf.} \cite[Theorem 1.9]{crew84}) on the structure of the maximal pro-$p$
 quotient of the \'etale fundamental group of
 a smooth and proper curve over an algebraically closed field of characteristic $p$.

When $\Sigma\neq\emptyset$ the very nature of wild ramification forces the local--local
$p$-divisible groups $\Gscr_n^{\ll}$ to grow rapidly (see \cite{GK}), and using work of Wintenberger \cite{Wintenberger},
we prove that when $\Gamma$ is a $p$-adic Lie group, the local-local Iwasawa--Dieudonn\'e module $\D^{\ll}$ is {\em not} finitely generated as a $\Lambda$-module.  In fact, 
the situation is even worse: as $F$ and $V$ are topologically nilpotent on $\D^{\ll}$,
it is naturally
a module over the ``Iwasawa--Dieudonn\'e'' algebra $\Lambda[\![F,V]\!]$ of formal power
series in commuting variables $F,V$ satisfying $FV=p$, $F\lambda = \sigma(\lambda)F$, and $V\lambda=\sigma^{-1}\lambda V$,
for $\lambda\in \Lambda$; using \cite{BooherCais}, we prove that $\D^{\ll}$ is not finitely generated 
as a module over {\em this} ring either!  Nevertheless, in \cite{BooherCais2},
Booher and the author investigate the structure of 
the $F$-torsion group schemes $\Gscr_n^{\ll}[F]$
in certain $\Gamma=\ZZ_p$-towers with $X_0=\PP^1$ and $\Sigma=\{\infty\}$.
Based on extensive computational evidence, and in perfect analogy with 
the asymptotic formula of Corollary \ref{thmI} in the unramified case, 
we formulate precise conjectures which imply that these $F$-torsion, local--local,
class group schemes grow in an astonishingly regular and predictable manner.
In forthcoming 
work with Booher, Kramer--Miller and Upton, we use ideas from Dwork theory to prove several of these conjectures,
which would seem to be totally out of reach of Iwasawa--theoretic methods ({\em i.e.} structure theory of $\Lambda$-modules).

Using our structure and control theorems for $\D^{\et}$, 
we prove that when $k$ is finite and
$\Gamma$ is a $d$-dimensional, torsion-free $p$-adic Lie group
equipped with its lower central $p$-series filtration $\{\Gamma_n\}$, 
the $p$-torsion of ``physical'' class groups in any ramified $\Gamma$-tower
of curves $\{X_n\}$ satisfies an Iwasawa-style asymptotic formula; see Corollary \ref{thmH}
for the precise statement.
When $\Gamma=\ZZ_p^d$, the analogue of our formula for $\Gamma$-extensions of number fields
was proved by Monsky \cite{MonskypRanks}, using the results \cite{CuocoMonsky}.
The proofs of both Monsky's formula and ours ultimately rest on 
knowing that the relevant ``$p$-torsion'' Iwasawa module is finitely generated,
and are similar in spirit.  In the abelian case $\Gamma=\ZZ_p^d$,
it is possible to obtain an exact asymptotic formula---as Monsky does---by appealing
to the theory of Hilbert--Kunz multiplicity initiated in \cite{MonskyHK}.
For general $\Gamma$, we obtain only asymptotic upper and lower bounds, of the same order of growth
but with possibly different constants, by appealing to \cite[Proposition 2.18]{EmertonPaskunas},
which relies on Venjakob's influential work \cite{Venjakob} on the structure of Iwasawa algebras
of $p$-adic Lie groups.

In the number field setting, and when $\Gamma=\ZZ_p^d$, 
Cuoco \cite{Cuoco} (for $d=2$), Cuoco and Monsky \cite{CuocoMonsky} 
and later Monsky \cite{Monsky2} \cite{Monsky3} \cite{Monsky4} \cite{MonskyFine} (for general $d$) established
increasingly more refined asymptotic formulae for the 
order of the full $p$-primary subgroup of the class group.
As noted by Li and Zhao \cite{Zhao}, their methods---which again 
ultimately rest upon the structure theory of Iwasawa modules and commutative algebra---carry over {\em mutatis mutandis} 
to the function field case.
Using an entirely different approach via $L$-functions, 
Wan has recently established a beautiful {\em exact} formula
 \cite[Theorem 1.3]{Wan} for the order of the full $p$-primary subgroup of $\Cl_{X_n}$
 in any $\Gamma=\ZZ_p^d$ tower of function fields over a finite field;
 in the number field setting, the analogous exact formula has been conjectured by 
 Greenberg (stated in \cite[p.~236]{CuocoMonsky}), and remains open.
 Wan also conjectures \cite[Conecture 5.1 (2)]{Wan} an analogous exact formula for 
 the order of the $p$-primary subgroup of the class group
 in {\em arbitrary} ({\em i.e.} not necessarily abelian) $\Gamma$-extensions
 of a function field over a finite field, and our Corollary \ref{thmH}
 provides new evidence for Wan's 
 conjecture.

In order to state our main results precisely, we must first fix some notation.
When $\Sigma\neq\emptyset$, we set $S:=\Sigma$, and when $\Sigma=\emptyset$,
we choose a closed point $x_0$ of $X_0$ and put $S:=\{x_0\}$.
For each $n$, we write $S_n:=(\pi_{n,0}^{-1}S)_{\red}$
for the reduced scheme underlying the (scheme-theoretic) fiber of $\pi_{n,0}:X_n\rightarrow X_0$,
so $S_0=S$.  We view $S_n$ as a {\em modulus} on $X_n$,
and we analyze the modules \eqref{ID} by first studying their analogues for the ``motivic'' {\em ray class groups} of conductor $S_n$.  
We form the generalized Jacobian $J_{n,S_n}$ of $X_n$ with modulus $S_n$
which, as $S_n$ is \'etale, is an extension of the usual Jacobian $J_{n}$ by a torus.
The inductive system $\Gscr_{n,S_n}:=\{J_{n,S_n}[p^n]\}_{n\ge 1}$ of $p$-power torsion group schemes
is then a $p$-divisible group, and for each $n$ we have two extensions of $p$-divisible groups
\begin{subequations}
\begin{equation}
    \xymatrix{
        0 \ar[r] & {\T_n} \ar[r] & {\Gscr_{n,S_n}} \ar[r] & {\Gscr_n} \ar[r] & 0
    }\label{seqA}
\end{equation}
\begin{equation}
    \xymatrix{
        0 \ar[r] & {\Gscr_n} \ar[r] & {\Gscr_{n,S_n}^{\vee}} \ar[r] &  {\Qscr_n} \ar[r] & 0
    }\label{seqB}
\end{equation}
\end{subequations}
in which $\T_n$ (respectively $\Qscr_n$) 
is a group of multiplicative (respectively \'etale) type, {\em i.e.}~isomorphic to a finite number
of copies of $\mu_{p^{\infty}}$ (respectively $\QQ_p/\ZZ_p$) over $\o{k}$.
These extensions are exchanged by duality, using the canonical autoduality of $\Gscr_n$
coming from the principal polarization on each Jacobian.
We then form the Iwasawa Dieudonn\'e modules 
\begin{equation}
        \D^{\star}_{S}:=\displaystyle\varprojlim_{n,\pi^*} \D(\Gscr_{n,S_n}^{\star})\ \text{for}\ \star\in \{\mult,\ll\}\ \text{and}\ 
        \D^{\et}_{S}:=\displaystyle\varprojlim_{n,\pi_*^{\vee}} \D((\Gscr_{n,S_n}^{\vee})^{\et})
        \label{DDef}
\end{equation}
using Picard functoriality for the first, and Albanese functoriality and duality for the second.
In this way, $\D^{\star}_S$ is naturally a (left) topological $\Lambda$-module with continuous, semilinear actions of $F$ and $V$.
As ramification in any $p$-group cover of curves in characteristic $p$ is necessarily {\em wild}, the
ramified case $\Sigma\neq\emptyset$ and unramified case $\Sigma=\emptyset$
are {\em radically} different, and in some ways our results in the unramified case are more satisfying, so we describe 
them first.

\begin{thmx}\label{thmA}
    Assume $\Sigma=\emptyset$ and let $g$ and $\gamma$ be the genus and $p$-rank of $X_0$, respectively.
   \begin{enumerate}
    \item The $\Lambda$-module $\D^{\star}_S$ is free
    of rank $\gamma$
    for $\star=\et,\mult$ and free of rank $2(g-\gamma)$ for $\star=\ll$.
    \label{thmA:str}
    
    \item \label{thmA:control} The canonical projection maps yield isomorphisms of $\Lambda_n$-modules
   \begin{equation*}    
                { \Lambda_n \tens_{\Lambda} \D^{\star}_S } \simeq \D(\Gscr_{n,S_n}^{\star})
                \quad\text{for}\ \star\in \{\mult,\ll\}\ \text{and}\quad
                {\Lambda_n \tens_{\Lambda} \D^{\et}_S } \simeq \D((\Gscr_{n,S_n}^{\vee})^{\et})
    \end{equation*}
   for all $n$, compatibly with $F$ and $V$.

    \item There are canonical, $\Lambda$-bilinear perfect pairings 
   \begin{equation*}
    (\cdot,\cdot):\D^{\mult}_S \times \D^{\et}_S \rightarrow \Lambda
    \quad \text{and}\quad (\cdot,\cdot):\D^{\ll}_S \times \D^{\ll}_S \rightarrow \Lambda 
    \end{equation*}    
   with respect to which $F$ and $V$ are adjoint, and which identify
   the each of $\D^{\et}_S$ and $\D^{\mult}_S$ with the $\Lambda$-dual $($Definition \ref{dualdef}$)$ of the other,
   and $\D^{\ll}_S$ with its own $\Lambda$-dual.

  \end{enumerate}
 \end{thmx} 
 
 From theorem \ref{thmA}, we deduce analogous structure and control theorems 
 for the Iwasawa--Dieudonn\'e modules $\D^{\star}$.  Here and in what follows,
 we put $I_n:=\ker\left( \Lambda\twoheadrightarrow \Lambda_n\right)$, and set $I:=I_0$.
  
 \begin{thmx}  \label{thmB}
    With the notation and hypotheses of Theorem \ref{thmA}:
 \begin{enumerate}
   \item\label{eta} 
   There is a canonical isomorphism of $\Lambda$-modules with $F$ and $V$ action
   \begin{equation*}
        \D^{\ll}\simeq \D^{\ll}_S
   \end{equation*}
   In particular, $\D^{\ll}$ is a free, self-dual $\Lambda$-module of rank $2(g-\gamma)$.
   For each $n$, the projection maps yield canonical isomorphisms of $\Lambda_n$-modules with $F$ and $V$ action
   \begin{equation*}
        \Lambda_n \tens_{\Lambda}\D^{\ll}\simeq \D(\Gscr_{n}^{\ll}).
   \end{equation*}
   
   \item\label{etb} Assume that the finite \'etale $k$-scheme $S_n=\pi_{n,0}^{-1}(S)$ splits completely over $k$ for all $n\ge 0$.
   Then there are canonical short exact sequences of $\Lambda$-modules with $F$ and $V$-action
   \begin{subequations}
        \begin{equation*}
            \xymatrix{
                0 \ar[r] & {\Lambda} \ar[r] & {\D^{\et}_S} \ar[r] & {\D^{\et}} \ar[r] & 0
            }
        \end{equation*}
        \begin{equation*}
            \xymatrix{
                0 \ar[r] & {\D^{\mult}} \ar[r] & {\D^{\mult}_S} \ar[r] & {I} \ar[r] & \D^{\mult,1} \ar[r] & 0
            }
        \end{equation*}
   \end{subequations}
   where $\D^{\mult,1}=\varprojlim^1_{n} \D^{\mult}_n$. 
   Here, $F=\sigma$ and $V=p\sigma^{-1}$ on $\Lambda$ 
   and 
   $F=p\sigma$ and $V=\sigma^{-1}$ on $I$.
   
   \item\label{etc} Under the hypothesis of \ref{etb}, for each $n$ there are canonical isomorphisms of $\Lambda_n$-modules
   \begin{equation*}
        \Lambda_n \tens_{\Lambda}\D^{\et}  \simeq \D(\Gscr_n^{\et}),\quad\text{and}\quad \Tor_1^{\Lambda}(\Lambda_n,\D^{\et}) \simeq W
   \end{equation*}
   \end{enumerate}
\end{thmx}

In general, the derived limit $\D^{\mult,1}$
appearing in the second exact sequence of Theorem \ref{thmB} \ref{etb}
seems rather mysterious; as a result, the kernel and cokernal 
of the canonical map $\Lambda_n \otimes_{\Lambda}\D^{\mult} \rightarrow \D_n^{\mult}$
are likewise mysterious; see Remark \ref{imrho} for more discussion.
However, when $k$ is algebraically closed and $L/K$ is the maximal unramified $p$-extension, 
we prove that both $\D^{\mult}$ and  $\D^{\mult,1}$ vanish, and thereby
give a 
new proof of Shafarevich's theorem 
\cite[\foreignlanguage{russian}{Tеорема }2]{Shafarevich}
on the structure of the maximal pro-$p$ quotient of the \'etale fundamental group of a projective curve
over an algebraically closed
field of characteristic $p>0$:

\begin{cor}\label{thmC}
    With the notation and hypotheses of Theorem \ref{thmA}, assume that $k$ is agebraically closed
    and that $\Gamma$ is the maximal pro-$p$ quotient of $\pi_1^{\et}(X_0)$.  Then $\D^{\mult}=0=\D^{\mult,1}$,
    and $I$ is a free $\Lambda$-module of rank $\gamma$; in particular, 
    $\Gamma$ is a free pro-$p$ group on $\gamma$ generators.
\end{cor}

\begin{remark}\label{unrkhyp}
    The existence of an everywhere unramified $\Gamma$-extension 
    $L/K$ with $\Gamma$ an infinite pro-$p$ group forces $k$ to be rather large.  
    For example, if $k$ is finitely generated (as a field) over $\FF_p$
    there are {\em no} geometric, everywhere unramified, infinite abelian extensions $L/K$
    whatsoever! See Theorem 2 (and {\em cf.}~Theorem 5) of \cite{KatzLang}.
    On the other hand, if $k$ is algebraically closed, then by Shafarevich's Theorem,
    if $\Gamma$ is
    {\em any} pro-$p$ group that can be (topologically) generated by at most $\gamma$
    generators, then 
    there exists an everywhere unramified $\Gamma$-extension
    of $K$.    
\end{remark}

\begin{remark}
    The hypothesis that $S_n$ split completely over $k$ in \ref{etb}--\ref{etc} of Theorem \ref{thmB}
    is needed in order to obtain the relatively simple descriptions of $\ker(\D^{\et}_S\rightarrow \D^{\et})$
    and $\coker(\D^{\mult}\rightarrow \D^{\mult}_S)$ implicit in \ref{etb}; without it, 
    these $\Lambda$-modules depend on the finer arithmetic properties of the \'etale schemes $S_n$,
    and seem to be considerably harder to describe.  This hypothesis may at first appear
    somewhat restrictive, but in practice it is rather innocuous. It is certainly verified
    if $k$ is algebraically closed.
    For general $k$, if the chosen point $x_0$ is $k$-rational and $K$ admits an 
    everywhere unramified geometric $\Gamma$-extension $L/K$ with $\Gamma$ an infinite abelian pro-$p$ group, 
    then there is such an extension $L'/K$ with $L'\otimes_k \o{k}\simeq L\otimes_k \o{k}$
    in which $x_0$ is split completely; see \cite[\S0]{KatzLang}. 
\end{remark}

Recall that the {\em order} $|G|$ of a finite $k$-group scheme $G$ is the $k$-dimension of
its Hopf algebra. 
If $\Gscr$ is a $p$-divisible group over $k$, and $J$
is any proper ideal of $\ZZ_p[\![F,V]\!]$ containing $p$
we define
\begin{equation*}
    \Gscr[J] := \bigcap_{f\in J} \ker(f: \Gscr\rightarrow \Gscr) \quad\text{(scheme-theoretic intersection)}.
\end{equation*}
This is a finite $k$-subgroup scheme of $\Gscr[p]$, so in particular is killed by $p$.
As a consequence of Theorem \ref{thmB} \ref{eta}, 
for any such $J$ we deduce an Iwasawa--style formula for the orders $|\Gscr^{\ll}_n[J]|$
of $J$-torsion local--local group schemes in unramified $p$-adic Lie extensions:

\begin{cor}\label{thmI}
    Let $\Gamma$ be a $p$-adic Lie group of dimension $d$ without $p$-torsion, and
    $\{X_n\}$ an \'etale $\Gamma$-tower of curves with $X_n$
    corresponding to the $n$-th subgroup in the lower central $p$-series of $\Gamma$ $($Definition \ref{pseries}$)$.
   Let $J$ be a proper ideal of $\ZZ_p[\![F,V]\!]$ containing $p$, 
    and write $\delta\in [0,d]$ for the {\em dimension}
    of the $\Lambda/p\Lambda$-module $\D^{\ll}/J\D^{\ll}$ in the sense of \cite[Definition 3.1]{Venjakob}.
    Then there exist real constants $\nu \ge \mu \ge \frac{1}{\delta!}$ such that
    \begin{equation*}
     \mu p^{\delta n} + O(p^{(\delta-1)n})  \le \log_p |\Gscr_n^{\ll}[J]| \le \nu p^{\delta n} + O(p^{(\delta-1)n}).
    \end{equation*}
    If $\Gamma=\ZZ_p^d$ is abelian, we may moreover take $\nu=\mu$.
\end{cor}

When $J=(p)$, $J=(F)$, or $J=(V)$, it is
straightforward to deduce an {\em exact} formula for $|\Gscr_n^{\ll}[J]|$
using the Riemann--Hurwitz and Deuring--Shafarevich formulae; see Remark \ref{llDS}.
For other choices of $J$, the group schemes $\Gscr_n^{\ll}[J]$ are
much more mysterious; for example, the integer $a_n:=\log_p|\Gscr_n^{\ll}[(F,V)]|$ 
is called the {\em $a$-number} of $X_n$, and is a subtle numerical invariant
that has been studied extensively in many different contexts \cite{CMHyper,VolochChar2,Fermat,ReBound,ElkinPriesanum1,ElkinCyclic,Suzuki,Dummigan,FermatHurwitz,Frei,BooherCais}.
Examples show that there can be no exact analogue of the Riemann--Hurwitz or Deuring--Shafarevich formulae
for the growth of $a$-numbers in (branched) Galois coverings of curves \cite[Example 7.2]{BooherCais}, 
so the existence of an asymptotic formula for the $a_n$ is surprising.
In this way, Corollary \ref{thmI} 
provides a fundamentally new kind of Iwasawa-theoretic result for $p$-adic
Lie extensions of function fields.

Let us now turn to the ramified case $\Sigma\neq \emptyset$.
Here, our hypothesis that every point of $\Sigma=S_0$ totally ramifies in $X_n$
means that $\pi_{n,0}$ induces an isomorphism of finite \'etale $k$-schemes 
$S_{n}\simeq S_0$ 
for all $n$, and we will henceforth make this identification,
writing $S$ for this common $k$-scheme. 
These identifications then induce isomorphisms $\T_n \simeq \T_0$
and $\Qscr_n\simeq \Qscr_0$ for all $n$, and we will similarly write $\T_S$ and $\Qscr_S$
for these common $p$-divisible groups. 
We will likewise simply write $\Gscr_{n,S}$ in place of $\Gscr_{n,S_n}$.

\begin{thmx}\label{thmD}
    Assume that $\Sigma\neq \emptyset$ and let $d:=\gamma + \deg S-1$ where 
    $\gamma$ is the $p$-rank of $X_0$.
    \begin{enumerate}
    \item $\D^{\mult}_S$ and $\D^{\et}_S$ are free $\Lambda$-modules of rank $d$.\label{thmD:str}

    \item The canonical projection maps yield isomorphisms of $\Lambda_n$-modules 
    \begin{equation*}
        \Lambda_n \tens_{\Lambda}\D^{\mult}_S  \simeq \D(\Gscr_{n,S}^{\mult})
        \quad\text{and}\quad
        \Lambda_n \tens_{\Lambda}\D^{\et}_S  \simeq \D((\Gscr_{n,S}^{\vee})^{\et})
    \end{equation*}
    for all $n$, compatibly with $F$ and $V$.\label{thmD:control}
    
    \item  There is a canonical, $\Lambda$-bilinear perfect pairings 
   \begin{equation*}
    (\cdot,\cdot):\D^{\mult}_S \times \D^{\et}_S \rightarrow \Lambda
    \end{equation*}    
   with respect to which $F$ and $V$ are adjoint, and which identifies
    each of $\D^{\et}_S$ and $\D^{\mult}_S$ with the  $\Lambda$-dual of the other.
    \label{thmD:duality}

\end{enumerate}    
\end{thmx}    

Analogously to Theorem \ref{thmB}, we have:

\begin{thmx}\label{thmE}    
With the notation and assumptions of Theorem \ref{thmD}:
\begin{enumerate}    
    \item \label{thmE:str}There is a canonical isomorphism of $\Lambda$-modules
    $$\D^{\et} \simeq \D^{\et}_S$$
    and a canonical short exact sequence of $\Lambda$-modules
    \begin{equation*}
        \xymatrix{
            0 \ar[r] & {\D^{\mult}} \ar[r] & {\D^{\mult}_S} \ar[r] & \D(\T_S)\ar[r] & 0   
        }\label{mulpart}
    \end{equation*}
    that are compatible with $F$ and $V$.  In particular, $\D^{\et}$ is free of rank $d$ over $\Lambda$.
\item \label{thmE:control} For each $n$, there are canonical short exact sequences $\Lambda_n$-modules with $F$ and $V$ action
    \begin{equation*}
        \xymatrix{
            0 \ar[r] & {\D(\Qscr_S)} \ar[r] & {\Lambda_n \tens_{\Lambda}\D^{\et}  } \ar[r] & {\D(\Gscr_n^{\et})} \ar[r] & 0
        }
    \end{equation*}
    \begin{equation*}
        \xymatrix{
            0 \ar[r] & {\displaystyle\frac{I_{n}}{I\cdot I_{n}}\tens_W\D(\T_S) } \ar[r] & {\Lambda_n \tens_{\Lambda}\D^{\mult} } \ar[r] & {\D(\Gscr_n^{\mult})} \ar[r] & 0
        }
    \end{equation*}

    \item \label{thmE:torus}There are canonical isomorphisms of $\Lambda$-modules with $F$ and $V$ 
    action
    \begin{equation*}
       \xymatrix{
       {\D(\Qscr_S) \simeq \coker\Big(W} \ar[r]^-{\Delta} & {\displaystyle\bigoplus_{s\in S} W(k(s))\Big)}
        } 
        \ \text{and}\ 
        \xymatrix@C=50pt{
       {\D(\T_S) \simeq \ker\Big(\displaystyle\bigoplus_{s\in S} W(k(s))} \ar[r]^-{\sum \Tr_{W(k(s))/W}} & {W\Big)}
        }
    \end{equation*}
    with $W(k(s))$ viewed as a $\Lambda$-module via the augmentation map $\Lambda\twoheadrightarrow W\hookrightarrow W(k(s))$.
    In the first $($respectively second$)$ identification, $F=\sigma$ $($resp. $F=p\sigma$$)$ on $W(k(s))$,
    and $V=p\sigma^{-1}$ $($resp. $V=\sigma^{-1}$$)$.

\end{enumerate}
\end{thmx}

In the spirit of Iwasawa theory, and in analogy with Theorem \ref{thmB}, one might expect (hope?) that $\D^{\ll}$
is finitely generated over $\Lambda$ in the ramified setting as well.
Unfortunately, this turns out {\em not} to be the case.  In fact, the influence of
wild ramification is so severe that the situation is even worse: 
Noting that $F$ and $V$ are topologically nilpotent on $\D^{\ll}$,
we may consider the $\Lambda$-module $\D^{\ll}$ as a left module over the larger ``Iwasawa--Dieudonn\'e'' ring $\Lambda[\![F,V]\!]$,
and in rather general situations it isn't finitely generated over {\em this} ring either:

\begin{thmx}\label{thmF}
    If $\Gamma$ is a $p$-adic Lie group and $\Sigma\neq\emptyset$,
    the $\Lambda[\![F,V]\!]$-module $\D^{\ll}$ is not  finitely generated.
\end{thmx}

To relate our work to the growth of $p$-primary components of ``physical'' class groups
in a given $\Gamma$-tower $\{X_n\}$ 
we write $(\cdot)^*:=\Hom_{\ZZ_p}(\cdot,\QQ_p/\ZZ_p)$ for the Pontryagin dual, 
and form the $\ZZ_p[\![\Gamma]\!]$-module 
$$M:=\varprojlim_{n} \Cl_{X_n}[p^{\infty}]^{*} = \varprojlim_{n} \Gscr^{\et}_{n}(k)^{*}=\varprojlim_n \Gscr_n(k)^*$$ 
with transition maps the Pontryagin duals of the pullback maps $\pi^*$ on $p$-divisible groups.
We write $M_W:=\Lambda \otimes_{\ZZ_p[\![\Gamma]\!]}M = W\otimes_{\ZZ_p}M $ for the (left) $\Lambda$-module
obtained by extension of scalars.

\begin{thmx}\label{thmG}
    Assume $|k|=p^r$ and set $\wp:=1-F^r$.  Then there is a canonical short exact sequence of $\Lambda$-modules
    \begin{equation*}
        \xymatrix{
            0 \ar[r] & {\D^{\et}} \ar[r]^-{\wp} & {\D^{\et}} \ar[r] & M_W \ar[r] & 0
        }.   
    \end{equation*}
    For each $n$, there are canonical short exact sequences of $\Lambda_n$-modules
    \begin{equation*}
          \xymatrix{
            0 \ar[r] & {W\tens_{\ZZ_p} \Qscr_S(k)^*} \ar[r] & {\Lambda_n \tens_{\Lambda} M_W  } \ar[r] & {W\tens_{\ZZ_p}\Gscr_n(k)^*}
            \ar[r] & 0
          }.
    \end{equation*}
\end{thmx}


\begin{cor}\label{thmH}
        Let $\Gamma$ be a torsion-free $p$-adic Lie group of dimension $d$ and
        $\{X_n\}$ a ramified $(${\em i.e.}~$\Sigma\neq\emptyset$$)$ $\Gamma$-tower of curves 
        over a finite field $k$ with $X_n$
    corresponding to the $n$-th subgroup in the lower central $p$-series of $\Gamma$.
        Let $\delta\in [0,d]$ be the dimension of the $\FF_p[\![\Gamma]\!]$-module $M/pM$.
        Then there exist real constants $\nu\ge\mu\ge \frac{1}{\delta!}$ with
        \begin{equation*}
            \mu p^{\delta n}  + O (p^{(\delta-1)n}) \le \log_p |\Cl_{X_n}[p]| \le \nu p^{\delta n}  + O (p^{(\delta-1)n}).
        \end{equation*}
        If $\Gamma=\ZZ_p^d$ is abelian, we may moreover take $\nu=\mu$.
\end{cor}







\section{Preliminaries}

\subsection{Iwasawa algebras and modules}

Here we record some basic facts and notation for completed group rings and modules over them,
using \cite{NSW} and \cite{Venjakob} as our guides.
Throughout, we fix an infinite pro-$p$ group $\Gamma$ equipped
with a countable basis $\Gamma=\Gamma_0\supsetneq \Gamma_1\supsetneq\cdots$
of the identity consisting of open normal subgroups.
As in \S\ref{intro}, we write $W:=W(k)$ for the Witt ring of $k$.

\begin{defn}
    For each $n$, we give $\Lambda_n:=W[\Gamma/\Gamma_n]$ (respectively $\Omega_n:=k[\Gamma/\Gamma_n]$) the $p$-adic  (resp. discrete) topology, and equip
    \begin{equation*}
        \Lambda:=\varprojlim_{n} W[\Gamma/\Gamma_n]\qquad\text{and}\qquad \Omega:=\varprojlim_n \Omega_n
    \end{equation*}
    with the inverse limit topology.
    We set $I_{n}:=\ker(\Lambda\twoheadrightarrow \Lambda_n)$, and write simply
    $I:=I_0$  for the augmentation ideal of $\Lambda$. 
\end{defn}
The $I_{n}$ form a decreasing chain of two-sided, closed ideals of $\Lambda$,
and each $I_n$ is generated (as either a left or right ideal) by all expressions of the form $\gamma-1$ with $\gamma\in \Gamma_n$.
Furthermore, the collection of ideals $\{p^m\Lambda+ I_n\}_{n,m\ge 0}$ form a fundamental system of
neighborhoods of  $0\in \Lambda$.



\begin{defn}\label{pseries}
Assume that $\Gamma$ is topologically finitely generated. The {\em lower central $p$-series of $\Gamma$} is defined recursively
by 
$$P_0:=\Gamma\quad\text{and}\quad P_{n+1}:=P_n^p [P_n,\Gamma],$$
where $P_n^p:=\langle g^p\ :\ g\in P_n\rangle$ is the subgroup of $P_n$ generated by all $p$-th powers, and $[P_n,\Gamma]$ is the commutator subgroup of $\Gamma$;
see Definition 1.15 and Corollary 1.20 of \cite{Dixon}.  
\end{defn}

We note that the $P_n$ form a basis of the identity in $\Gamma$ 
consisting of open normal subgroups \cite[Proposition 1.16]{Dixon}.

\begin{lem}\label{pcoker}
    Assume that $\Gamma$ is topologically finitely generated and let $\Gamma_n:=P_n$ for all $n$.
    Then $pI_n \subseteq I_{n}I + I_{n+1}$.
\end{lem}

\begin{proof}
    As a $\Lambda$-module, $I_n$ is generated by elements of the form $u:=x-1$ with $x\in \Gamma_n$,
    so it suffices to prove that $pu$ lies in $I_nI+I_{n+1}$.  
    We have 
    $$x^p -1= (1+u)^p - 1 = pu + u^2 w$$
    for some $w\in \Lambda$.  Now $x^p\in \Gamma_{n+1}$, so $x^p - 1 \in I_{n+1}$
    and clearly $u^2w \in I_n^2 \subseteq I_n I$; the result follows.
\end{proof}

\begin{lem}
    The ring $\Lambda$ is local, with unique maximal ideal $p\Lambda+I$.
    In particular, for every $n$ the quotient $\Lambda_n$ is local with unique maximal ideal $p\Lambda_n + I_{\Gamma/\Gamma_n}$. 
\end{lem}

\begin{proof}
    This follows immediately from \cite[Proposition 5.2.16 (iii)]{NSW}
\end{proof}

\begin{prop}
    If $\Gamma$ is $p$-adic analytic, then $\Lambda$ is both left and right noetherian.
     If moreover $\Gamma$ has no elements of finite order, then $\Lambda$
    has no zero divisors.
\end{prop}

\begin{proof}
    The first assertion is \cite[V, 2.2.4]{Lazard}, while the second is \cite[Theorem 1]{Neumann}.
\end{proof}

By a left (respectively right) $\Lambda$-{\em module}, we will always mean a separated, topological left (resp. right) $\Lambda$-module.
If $M$ is any left $\Lambda$-module, we define $M^{\op}$ to be the right $\Lambda$-module
with the same underlying abelian group as $M$ and right $\Gamma$-action given by $m\cdot\gamma:=\gamma^{-1}m$
for $m\in M$.  This enables us to define a duality functor from the category of left $\Lambda$-modules to itself:

\begin{defn}\label{dualdef}
    If $M$ is any left $\Lambda$-module, 
    We define the {\em dual} of $M$ to be the left $\Lambda$-module
    \begin{equation*}
        M^{\vee}:=\Hom_{\Lambda}(M^{\op},\Lambda).
    \end{equation*}
    This is
     the abelian group of all continuous homomorphisms of right $\Lambda$-modules $\varphi:M^{\op}\rightarrow \Lambda$
     with the left $\Lambda$-module structure 
     induced by the canonical 
    $\Lambda$-$\Lambda$-bimodule structure on $\Lambda$: $(\lambda\cdot\varphi)(m):=\lambda\varphi(m)$.
\end{defn}

\begin{remark}
Since $\Omega$ admits a homomorphism to the field $k$, 
both $\Omega$ and $\Lambda$ have {\em Invariant Basis Number} \cite[Definition 1.3, Remark 1.5]{LamLec},
and the {\em rank} of a free module is well-defined.
\end{remark}

\subsection{Towers of modules}

In this section, we generalize the commutative algebra formalism 
developed in \cite[\S3.1]{CaisHida1} for certain projective systems of 
$\Lambda$-modules when $\Gamma=\ZZ_p$ to the case of arbitrary pro-$p$ groups $\Gamma$.
This machinery is at the heart of our proofs of Theorems \ref{thmA} and \ref{thmD}.

\begin{defn}
    A {\em $\Gamma$-tower of $W$-modules} consists of the following data:
    \begin{enumerate}
        \item For each nonnegative integer $n$, a left $\Lambda_n$-module $\M_n$.
        \item For each pair of integers $n\ge m\ge 0$, a $\Gamma$-equivariant
        map of $W$-modules $\rho_{n,m}: \M_n\rightarrow \M_m$.
    \end{enumerate}
    Given a $\Gamma$-tower of $W$-modules $\{\M_n,\rho_{n,m}\}$, we write simply
    $\M:=\varprojlim_{n} \M_n$ for the projective limit
    taken with respect to the maps $\rho_{n,m}$; it is naturally 
    a left $\Lambda$-module.
\end{defn}

\begin{prop}\label{prop:NCfree}
    Let $\{M_n, \rho_{n,m}\}$ be a $\Gamma$-tower of $W$-modules with $M_n$
    a finite and free $W$-module for all $n$.
    Assume that the following conditions hold for all $n \ge 0$:
    \begin{enumerate}
        \item $\o{\M}_n:=k\otimes_W M_n$ a free $\Omega_n$-module 
        of rank $d$ that is independent of $n$.\label{Ass1}
        \item For all $m\le n$, the map 
        \begin{equation*}
            \xymatrix{
                {\o{\rho}_{n,m}:\o{\M}_n }\ar[r] & {\o{\M}_m}
                }
        \end{equation*}
        induced from $\rho_{n,m}$ by reduction modulo $p$ is surjective.
    \end{enumerate}
    Then for all $n\ge 0$:
    \begin{enumerate}
        \item $\M_n$ is a free $\Lambda_n$-module of rank $d$.\label{MUfree}
        \item The induced maps of left $\Lambda_m$-modules 
        \begin{equation}
            \xymatrix{
                { \Lambda_m \tens_{\Lambda_n} \M_n}  \ar[r] & {\M_m}
                }\label{eq:natmap}
        \end{equation}
        are isomorphisms for all $m\le n$.\label{MUisom}
    \end{enumerate}
    Moreover, $\M$ is a finite and free $\Lambda_W(\Gamma)$-module of rank $d$,
    and for each $n$ the canonical map $\Lambda_n \otimes_{\Lambda} M\rightarrow M_n$
    is an isomorphism of $\Lambda_n$-modules.
\end{prop}

\begin{proof}
    Fix $n\ge 0$ and choose $m_1,\ldots, m_d \in \M_n$ whose images in 
    $\o{\M}_n$ freely generate $\o{\M}_n$ as a left $\Omega_n$-module.
    The images  in $\o{\M}_n$  of the $[\Gamma:\Gamma_n] d$ elements $g m_i$ for $g\in \Gamma/\Gamma_n$ and $1\le i\le d$
   are then a basis for $\o{\M}_n$ as a $k$-vector space, and the free
    $W$-module $\M_n$ must have rank $[\Gamma:U]d$.  By Nakayama's lemma (in the usual commutative case),
    the set $\{gm_i\ :\ g\in \Gamma/\Gamma_n,\ 1\le i\le d\}$ is then a minimal set of generators
    of the free $W$-module $\M_n$ whose cardinality is equal to the $W$-rank of $\M_n$.
    It follows that this set freely generates $\M_n$ as a $W$-module, and hence that
    $\{m_i\ :\ 1\le i\le d\}$ freely generates $\M_n$ as a left $\Lambda_n$-module.
    
    Reducing the map \eqref{eq:natmap} modulo the ideal $(p)$
    yields a map of left $\Omega_m$-modules 
    \begin{equation}
        \xymatrix{
      {\Omega_m \tens_{\Omega_n} \o{\M}_n} \ar[r] & {\o{\M}_m}
      }\label{eq:redmodp}
    \end{equation}
    through which the surjective map $\o{\rho}_{n,m}$ factors,
    whence \eqref{eq:redmodp} is surjective.
    As $(p)$ is contained in the radical of $\Lambda_m$, it
    follows
    from Nakayama's Lemma \cite[(4.22) (3)]{Lam}
    that \eqref{eq:natmap} is surjective as well.
    By our assumption \ref{Ass1}, the map \eqref{eq:natmap} is then a surjective
    map of free $\Lambda_m$-modules of the same rank,
    whence it must be an isomorphism.\footnote{Indeed,
        let $R$ be any local ring with residue field $k$ 
        and $\rho:M\rightarrow M'$
        a surjective map of finite and free left $R$-modules of the same rank.
        As $M'$ is free, the surjection $\rho$ splits giving an isomorphism
        of left $R$-modules $M\simeq M'\oplus \ker(\rho)$.
        Applying $k\otimes_R (\cdot)$ to this isomorphism and using the fact that 
        $M$ and $M'$ are free of the same finite rank forces $k\otimes_R \ker(\rho) = 0$.
        The splitting of $\rho$ yields a surjection $\M\twoheadrightarrow \ker(\rho)$,
        whcnce $\ker(\rho)$ is finitely generated as $M$ is, and  
        Nakayama's Lemma \cite[p(4.22) (2)]{Lam} then gives $\ker(\rho)=0$, as desired.
    }
    
    For any $m\le n$, the kernel of the canonical
    surjection $\Lambda_n\twoheadrightarrow \Lambda_m$
    is contained in the radical of $\Lambda_n$.
    It follows from \ref{MUfree}--\ref{MUisom} and Nakayama's Lemma that any lift to $\M_n$
    of a $\Lambda_m$-basis of $\M_m$ is a $\Lambda_n$-basis of $\M_n$.
    We may therefore choose, for each $n\ge 0$, a basis
    $e_{n,1},\ldots,e_{n,d}$ of the free $\Lambda_n$-module $\M_n$
    with the property that $\rho_{n,m}: \M_n\rightarrow \M_m$
    carries $e_{n,i}$ to $e_{m,i}$ for all $m\le n$ and $1\le i\le d$.
    These choices yield isomorphisms of left $\Lambda_n$-modules 
    $\Lambda_n^{\oplus d}\simeq \M_n$ for all $n$ that are compatible with change in $n$
    via the maps $\rho_{n,m}$.  Passing to projective limits, 
    we conclude that $\M = \varprojlim_{n} \M_n$
    is free of rank $d$ over $\Lambda_W(\Gamma)$.
    The canonical map $\Lambda_n\otimes_{\Lambda} M\rightarrow M_n$
    is then a surjective map of free $\Lambda_n$-modules
    of the same rank, whence it is an isomorphism.
\end{proof}

We next wish to investigate duality for towers of $\Gamma$-modules.

\begin{prop}\label{prop:NCDual}
    Let $\{M_n, \rho_{n,m}\}$ and $\{M'_n, \rho'_{n,m}\}$ be two $\Gamma$-towers of $W$-modules 
    satisfying the hypotheses of Proposition \ref{prop:NCfree}.  Suppose that for each $n$
    there are $W$-bilinear perfect pairings
    \begin{equation}
        \xymatrix{
            {\langle \cdot, \cdot \rangle_n : M_n \times M_n'} \ar[r] & W
        }
    \end{equation}
    with respect to which $\gamma$ and $\gamma^{-1}$ are adjoint for all $\gamma\in \Gamma$, and which
    satisfy 
    \begin{equation*}        
            \langle \rho_{n,m}x , \rho'_{n,m}y\rangle_m = \sum_{\gamma \in \Gamma_m / \Gamma_n} \langle x , \gamma^{-1} y\rangle_n
    \end{equation*}
    for all $m\le n$.  Then for each $n$, the pairings 
    $\xymatrix{{(\cdot,\cdot)_n: M_n \times M_n'} \ar[r] & {\Lambda_n}}$ defined by 
    \begin{equation}
        (x,y)_n := \sum_{\gamma\in \Gamma/\Gamma_n} \langle x,\gamma^{-1}y\rangle_n \cdot \gamma
    \end{equation}
    compile to induce isomorphisms of left $\Lambda$-modules
    \begin{equation}
        \xymatrix{ 
           M \ar[r]^-{\simeq} & (M')^{\vee}
        }\quad\text{and}\quad
        \xymatrix{ 
           M' \ar[r]^-{\simeq} & M^{\vee}
        }
    \end{equation}
\end{prop}

\begin{proof}
    The proof of \cite[Lemma 3.4]{CaisHida1} carries over {\em mutatis mutandis}
    to the present setting.
\end{proof}

We end this section with a variant of the Mittag--Leffler criterion \cite[\href{https://stacks.math.columbia.edu/tag/0598}{Tag0598}]{stacks-project},
which will enable us to show that certain (derived) inverse limits of $\Gamma$-towers vanish:

\begin{lem}\label{limitvanishing}
    Let $R$ be a commutative noetherian ring, complete with respect to an ideal $J$
    that is contained in the radical of $R$,
    and $\{M_n\}_{n}$ an inverse system of finitely generated $R$-modules.
    If every transition map $\rho:M_{n+1}\rightarrow M_n$ has
    image contained in $JM_n$, then 
    $\varprojlim_n M_n = \varprojlim^1_n M_n = 0$.
\end{lem}

\begin{proof}
    As $R$ is noetherian and $J$-adically complete and each $M_n$ is finitely generated, 
    $M_n$ is also $J$-adically complete \cite[Theorem 8.7]{Matsumura}.
    Since $J\subseteq \Rad R$, we likewise know that $M_n$ is $J$-adically separated
    by the Krull Intersection Theorem \cite[Theorem 8.10 (i)]{Matsumura}.
    By hypothesis, for all $r\ge 0$,  we have that the image of $\rho^r: M_{m+r}\rightarrow M_{m}$
    is contained in $J^{r}M_m$, so if $\{x_n\}_n \in \varprojlim_n M_n$, we 
    have $x_n \in \cap_{r\ge 1} J^r M_n =0$.
On the other hand, by \cite[\href{https://stacks.math.columbia.edu/tag/091D}{Tag 091D}]{stacks-project},
the derived limit $\varprojlim^1_n M_n$ is isomorphic to the cokernel of the map
\begin{equation}
    \xymatrix{
        {\tau: \prod_n M_n} \ar[r] & {\prod_n M_n}
        }\quad\text{given by}\quad \tau(\alpha_n):= (\alpha_n - \rho(\alpha_{n+1}))
        \label{lim1desc}
\end{equation}
We claim that $\tau$ is
surjective: given any $(\beta_n)_n\in \prod_n M_n$, for each fixed $n$ we set $\alpha_n:=\sum_{\ell\ge n} \rho^{\ell-n}(\beta_{\ell})$.
By hypothesis, $\rho^{\ell-n}(\beta_{\ell}) \in J^{\ell-n} M_n$, so as $M_n$ is $J$-adically complete,
this sum converges for each $n$, and  by construction we have $\tau((\alpha_n)_n)=(\beta_n)_n$.
Thus, $\varprojlim^1_n M_n = 0$
\end{proof}

\subsection{Generalized Jacobians}

In this section, we fix
 a smooth, proper and geometrically connected curve $X$ over $k$.

\begin{defn}
    A {\em modulus} $\m$ on $X$ is an effective Cartier divisor.
    If $\deg_k(\m)>1$, we denote by $X_{\m}$ the projective and geometrically integral
    curve associated to $X$ and the modulus $\m$ as in \cite[IV \S1 No.~4]{Serre}.  If $\deg_k(\m)\le 1$,
    we set $X_{\m}:=X$.  In all cases,
    we write $\nu:X\rightarrow X_{\m}$
    for the canonical map, via which $X$ is the normaliztion of $X_{\m}$.
\end{defn}

When $\deg_k(\m)>0$, we note that $X_{\m}$ has a distinguished $k$-rational point $x:\Spec k\rightarrow  X_{\m}$ with the
property that $X_{\m}$ is $k$-smooth outside of $x$, the divisor $\m$ on $X$
is the pullback of $x$ along $\nu$, and $\nu:X\rightarrow X_{\m}$
is initial among all $k$-maps $X\rightarrow Y$ to $k$-schemes $Y$ such that $\m$
scheme-theoretically factors through $Y(k)$.  It follows at once that for any $\m' \ge \m$,
the normalization map $\nu': X\rightarrow X_{\m'}$ factors uniquely through $\nu:X\rightarrow X_{\m}$.

\begin{lem}\label{lem:dualizingdesc}
    There is a functorial identification $\omega_{X_{\m}/k}\simeq \nu_*\Omega^1_{X}(\m)$ of the relative dualizing sheaf
    of $X_{\m}$ with the pushforward of the sheaf of relative differential forms on $X$ with poles no worse than $\m$.
    In particular, $X_{\m}$ has (arithmetic) genus $g_{X_{\m}}:=\dim_k H^0(X,\Omega^1_{X/k}(\m))$.
\end{lem}

\begin{proof}
This follows by Galois descent from Rosenlicht's explicit description \cite[Theorem 5.2.3]{GDBC}
of the relative dualizing sheaf of a proper reduced curve over an algebraically closed field.
\end{proof}

\begin{defn}
    For $n\in \ZZ$, we write $J_{X,\m}^{n}:=\Pic^n_{X_{\m}/k}$ for the connected component of the 
    relative Picard scheme $\Pic_{X_{\m}/k}$ classifying degree-$n$ line bundles on $X_{\m}$,
    and write simply $J_{X,\m}:=\Pic^0_{X_{\m}/k}$ for the {\em generalized Jacobian} of $X$ relative to $\m$.
    We will drop $\m$ from the notation when $\m=0$. 
\end{defn}

The $k$-group structure on $\Pic_{X_{\m}/k}$ (functorially induced by tensor product of line bundles)
makes $J_{X,\m}$ into a commutative $k$-group scheme
and 
$J_{X,\m}^n$ into a principal homogeneous space for $J_{X,\m}$ for all $n$.
The structure of $J_{X,\m}$ can be made quite explicit \cite[\S9.2]{BLR}:

\begin{prop}
    The commutative $k$-group scheme $J_{X,\m}$ is smooth over $k$ of dimension $g_{X_{\m}}$.  
    Pullback of line bundles along $\nu: X\rightarrow X_{\m}$ induces a surjective map of 
    commutative $k$-group schemes $J_{X,\m}\twoheadrightarrow J_{X}$ with kernel 
    that is canonically a product $U_{\m}\times T_{\m}$ of a unipotent group $U_{\m}$ and a torus $T_{\m}$.
    Furthermore, 
    there is a canonical 
    identification $T_{\m} \simeq \Res_{\m_{\red}/k} \Gm / \Gm$ of $T_{\m}$ with the quotient of the 
    Weil restriction $\Res_{\m_{\red}/k} \Gm$ by the diagonally embedded copy of $\Gm$, and
    $U_{\m}=0$ if and only if $\m=\m_{\red}$ is reduced.
\end{prop}

We now turn to the functoriality of $J_{X,\m}$ in $X$.  Fix a modulus $\m$ on $X$ supported on a finite set $S\subseteq X$
of closed points of $X$, and recall that for any $\m' \ge \m$, pullback of line
bundles along $X_{\m}\rightarrow X_{\m'}$ yields a canonical surjection of $k$-groups  $J_{X,\m'}\rightarrow J_{X,\m}$
with affine kernel.

\begin{prop}\label{GJfunc}
  Let $\pi:Y\rightarrow X$ be a finite and generically \'etale map of proper, smooth, and geometrically connected curves over $k$.
  For any modulus $\m$ on $X$ that is supported on a finite set of closed points $S\subseteq X$, there exists a unique
  maximal (respectively minimal) modulus $\n$ (respectively $\n'$) on $Y$ that is supported on $\pi^{-1}S$, 
  and canonical homomorphisms of $k$-group schemes
  \begin{equation*}
    \xymatrix{ 
        {\pi^* : J_{X,\m}} \ar[r] & {J_{Y,\n}}
        }\ and
        \xymatrix{ 
        {\pi_* : J_{Y,\n'}} \ar[r] & {J_{X,\m}}
        }
  \end{equation*}
  making the diagrams
  \begin{equation}
    \xymatrix{
        {J_{X,\m}}\ar[d] \ar[r]^-{\pi^*} & {J_{Y,\n}}\ar[d]\\
        {J_{X}} \ar[r]_-{\pi^*} & {J_{Y}} 
        }
    \quad
    \xymatrix{
        {J_{X,\m}}\ar[d] \ & {J_{Y,\n'}}\ar[d]\ar[l]_-{\pi_*} \\
        {J_{X}} & {J_{Y}}\ar[l]^-{\pi_*} 
        }\label{PicAlb}
  \end{equation}
  commute, where $\pi^*=\Pic^0(\pi) : J_X\rightarrow J_Y$ and $\pi_*=\Alb(\pi): J_Y\rightarrow J_X$
  are the usual maps on Jacobians associated to the finite (flat!) map $\pi: Y\rightarrow X$ 
  by Picard (pullback of line bundles) and Albanese (norm of line bundles) functoriality.
  Considering $S$ and $\pi^{-1}S$ as reduced Cartier divisors on $X$ and $Y$,
  these maps induce morphisms $\pi^*: J_{X}(S)\rightarrow J_{Y}(\pi^{-1}S)$
  and $\pi_*: J_{Y}(\pi^{-1}S)\rightarrow J_{X}(S)$ of maximal semiabelian quotients
  whose composition $\pi_* \pi^* = \deg(\pi)$ is multiplication by $\deg(\pi)$.
  If moreover $\pi$ is generically Galois with group $G$, then $\pi^*\pi_* = \sum_{g\in G} g^*$.

  The respective restrictions of $\pi^*$ and $\pi_*$ to the tori $T_{\m}$ and $T_{\n'}$ are induced
  by 
  the unique descents to $k$ of the $\o{k}$ maps 
  \begin{equation*}
        \xymatrix{
            {\Res_{S/k}\Gm \times_k \o{k} = \displaystyle\prod_{s\in S_{\o{k}}} \Gm }
            \ar@<2pt>[r]^-{\pi^*} & {\displaystyle\prod_{s\in S_{\o{k}}}\ \ \displaystyle\prod_{t\in \pi_{\o{k}}^{-1}(s)} \Gm =
            \Res_{\pi^{-1}S/k} \Gm \times_k \o{k}} \ar@<2pt>[l]^-{\pi_*}
    }
  \end{equation*}
  given on the factor with index $s\in S_{\o{k}}$ by the diagnonal $\pi^* = \Delta: \Gm \rightarrow  \prod_{t\in \pi_{\o{k}}^{-1}(s)} \Gm$
  and the norm map $\pi_* = \Nm: \prod_{t\in \pi_{\o{k}}^{-1}(s)} {\Gm} \rightarrow \Gm$ functorially determined on $R$-valued points
  for any $\o{k}$-algebra $R$ by $\Nm((a_t)_t):=\prod_t a_t^{e_t}$, with $e_t$ is the ramification degree of $\O_{X,s}\rightarrow \O_{Y,t}$.
\end{prop}

\begin{proof}
    We may assume that $\deg_k \m > 1$.  As $\pi: Y\rightarrow X$ certainly carries $\pi^{-1}S$ into $S$, 
    it follows from the universal property of $\nu:Y\rightarrow Y_{\n}$ that in order for the composite $\nu\circ \pi: Y\rightarrow X \rightarrow X_{\m}$ to factor through $Y_{\n}$, it is necessary and sufficient that the closed subscheme $\n$ of $Y$ scheme-theoretically factor through 
    the distinguished point $x\in X(k)$.  Writing $\I\subseteq \O_X$ for the ideal sheaf of $\m$ and $\J\subseteq \O_Y$ for the ideal sheaf of $\n$, such factorization holds if and only if $\J\subseteq f^*\I$ inside $\O_Y$, and $\J=f^*\I$ is patently maximal 
    with respect to this condition.  Assuming that $\n \le f^*\m$, we thus obtain a canonical map $\pi: Y_{\n}\rightarrow X_{\m}$
    whose normalization {\em is} $\pi:Y\rightarrow X$, and we denote by $\pi^*: J_{X,\m}\rightarrow J_{Y,\n}$ the map
    induced by pullback of line bundles along $\pi: Y_{\n}\rightarrow X_{\m}$.  It is clear from this construction that 
    the first diagram in \eqref{PicAlb} commutes.

    Denote by $\theta_{X,\m} : X - S\rightarrow J^1_{X,\m}$ the canonical map of $k$-schemes
    functorially determined by sending a $Z$-valued point $P$ of $X_Z-S_Z$ to the inverse
    ideal sheaf $\O(P)$ associated to the closed subscheme $P$ of $(X_{\m})_Z$.  
    Consider the map of $k$-schemes 
    \begin{equation}
        \xymatrix{
            {\rho: Y - \pi^{-1}S} \ar[r]^-{\pi} & {X-S} \ar[r]^-{\theta_{X,\m}} & {J_{X,\m}^1} 
            }.\label{YJcomp}
    \end{equation}
    By the universal mapping
    property of generalized Jacobians \cite[V \S4 No.~22, Proposition 13]{Serre}, there is a modulus $\n'$ on $Y$ with support
    $\pi^{-1}S$, a map of $k$-schemes $\rho^1_{\n'}: J_{Y,\n'}^1\rightarrow J_{X,\m}^1$, and a
    map of $k$-group schemes $\rho_{\n'}: J_{Y,\n'}\rightarrow J_{X,\m}$ such that $\rho^1_{\n'}$
    is equivariant with respect to $\rho_{\n'}$ and $\rho = \rho^1_{\n'}\circ \theta_{Y,\n'}$.
    Each such modulus uniquely determines $\rho_{\n'}$ and $\rho^1_{\n'}$, and there is a unique
    minimal modulus verifying these conditions.  We write simply $\pi_*: J_{Y,\n'}\rightarrow J_{X,\m}$
    for the map $\rho_{\n'}$ associated to the minimal modulus $\n'$. 
    As pullback of line bundles along $\nu$ preserves degree (see, e.g \cite[\S9.1]{BLR}),
    for any $n\in \ZZ$ we have a canonical map of $k$-schemes $J_{Y,\n'}^n\rightarrow J_{Y}^n$
    whose fibers are affine; as $J_X^n$ is proper, it follows that the composite
    \begin{equation}
        \xymatrix{
            {J_{Y,\n'}^n} \ar[r]^-{\pi_*} & {J_{X,\m}^n} \ar[r]^-{\nu^*} & {J_X^n}
        }\label{trgen}
    \end{equation}
    factors through $J_{Y,\n'}^n \rightarrow J_Y^n$ via a unique map $\mu^n: J_Y^n\rightarrow J_X^n$
    which must moreover be a map of $k$-groups when $n=0$.
    Consider the resulting diagram
    \begin{equation*}
        \xymatrix{
            {Y - \pi^{-1}S} \ar@/^20pt/[rrr]^-{\pi} \ar[r]_-{\theta_{Y,\n'}}\ar@{^{(}->}[d] & {J_{Y,\n'}^1} \ar[d]^-{\nu^*}  \ar[r]_-{\rho^1_{\n'}} & J_{X,\m}^1 \ar[d]_-{\nu^*} & 
            X-S \ar[l]^-{\theta_{X,\m}}\ar@{^{(}->}[d]  \\
            {Y} \ar[r]_-{\theta_Y} & {J_Y^1} \ar[r]_-{\mu^1} & {J_X^1} & X \ar[l]^-{\theta_X}
        }
    \end{equation*}
    in which the middle square commutes by the very definition of $\mu^1$
    and the outer squares are readily seen to commute from the definition of $\theta_{\star}$
    and the fact that the normalization maps are isomorphisms over the smooth locii.
    As the top region commutes by the very definition of $\rho^1_{\n'}$, 
    we conclude that $\mu^1\circ \theta_Y$ and 
    $\theta_X\circ \pi$ agree on the dense open subscheme $Y-\pi^{-1}S$ of $Y$; since $J_X^1$ is separated, we deduce
    that they agree on $Y$.  By the uniqueness aspect of $\rho^1$, it follows that $\rho^1=\mu^1$
    giving the commutativity of the second diagram in \eqref{PicAlb}.
    
    To prove that the restrictions of the maps we have constructed to the tori
    $T_{\m}$ and $T_{\n'}$ have the specified descriptions, we may assume that $k$
    is algebraically closed.  As two maps from a reduced $k$-scheme to a separated $k$-scheme agree
    if and only if they agree on $k$-points, it suffices to check that the proposed descriptions
    hold at the level of $k$-points, which we now do.

    Let $P,Q$ be two $k$-points of $Y$ not in $\pi^{-1}S$
    and let $\L_Y$ (respectively $\L_X$) be the line bundle on $Y_{\n'}$ (resp. $X_{\m})$ 
    corresponding to the degree zero Cartier divisor $Q-P$
    (respectively $\pi(Q)-\pi(P)$).
    It follows easily from the universal property of $\rho^1_{\n'}$ and $\rho_{\n'}$
    that $\pi_* \L_Y = \L_X$, and one concludes that if $D = \sum n_Q Q$ is {\em any} degree zero Cartier
    divisor on $Y$ supported away from $\n'$, then $\pi_*\L(D)= \L(\pi(D))$, where $\pi(D)=\sum n_Q \pi(Q)$
    is the direct image of $D$.  Thus, if $g\in k(Y)$ is any function, we conclude that 
    \begin{equation*}
        \pi_*(\L(\div(g))) = \L(\pi(\div(g))) = \L(\div(\Nm_{\pi}(g))),
    \end{equation*}
    where $\Nm_{\pi}: k(Y)\rightarrow k(X)$ is the norm function associated to the extension of function fields $k(Y)/k(X)$.
    {\em cf.}~the proof of \cite[III \S1 No.~2, Prop. 4]{Serre}.
    Choose a $k$-valued point $\xi$ of $T_{\n'} = \prod_{s\in S}\prod_{t\in \pi^{-1}(S)} \Gm/\Gm$, and
    let $(a_t)\in \prod_{s\in S}\prod_{t\in \pi^{-1}(S)} k^{\times} $
    be any representative of it.  Let $g \in k(Y)$ be any function with $\ord_t(a_t - g) > 0$ at every point $t\in \pi^{-1}S$.
    The image of $\xi$ under the inclusion $T_{\n'}(k)\rightarrow J_{Y,\n'}(k)$ is the line bundle $\L(\div(g))$
    on $Y_{\n'}$, with image under $\pi_*$ the line bundle $\L(\div(\Nm_{\pi}(g)))$ on $X(\m')$.
    For any point $s\in S$, the function $g$ is a local unit at every $t\in \pi^{-1}(s)$ by construction, and we have
    \begin{equation*}
        \Nm_{\pi}(g) = \prod_{t\in \pi^{-1}(s)} \Nm_{\wh{\O}_{Y,t}/\wh{\O}_{X,s}}(g) \equiv \prod_{t\in \pi^{-1}(s)} a_t^{e_t} \pmod{\p_s}
    \end{equation*}
    with $\p_s$ the maximal ideal of $\O_{X,s}$ and $e_t$ the ramification index of $\O_{X,s}\rightarrow \O_{Y,t}$.
    The given description of $\pi_*$ follows.  Likewise, if $h\in k(X)$ then $\pi^*\L(\div_X(h))=\L(\pi^*\div_X(h))=\L(\div_Y(\pi^*h))$,
    so if $a_s\in k^{\times}$ and $h$ satisfies $\ord_s(a_s-h)>0$ then $\ord_t(a_s-\pi^*h)>0$ for every $t\in \pi^{-1}s$
    since the pullback map $\pi^*: k(X)\rightarrow k(Y)$ is a map of $k$-algebras. The given description of $\pi^*$
    follows.
    
    Finally, $\pi^*$ and $\pi_*$ induce maps on maximal semiabelian quotients as
    any $k$-homomorphism from a unipotent group to a semiabelian variety is zero, and the fact that these induced maps 
    satisfy $\pi_*\pi^*=\deg(\pi)$ and (when $\pi$ is genericlaly Galois with group $G$) $\pi^*\pi_* = \sum_{g\in G} g^*$ follows from the corresponding fact on (usual) Jacobians and the given descriptions
    of these maps on torii.
\end{proof}

\begin{defn}\label{GenpdivDef}
    For a smooth, proper, and geometrically connected curve $X$ over $k$ and a modulus $\m$ on $X$,
    we write $\Gscr_{X,\m}:=\{J_{X,\m}[p^n]\}_{n\ge 0}$ for the inductive system of 
    kernels of multiplication by $p^n$ on $J_{X,\m}$.
\end{defn}

\begin{cor}\label{GJstr}
    Assume that $\m$ is reduced, and write $\T_{\m}:=\Res_{\m/k}(\mu_{p^{\infty}})/\mu_{p^{\infty}}$
    for the multiplicative type $p$-divisible group given by the quotient of the Weil restriction its diagonally embedded copy of $\mu_{p^{\infty}}$.
    Then $\Gscr_{X,\m}$ is a $p$-divisible group, and is canonically an extension of $\Gscr_X=J_X[p^{\infty}]$ by $\T_{\m}$.
\end{cor}

\begin{proof}
    As $\Res_{\m/k}(\cdot)$ is right adjoint to base change $(\cdot)\otimes_k \m$, it ls left exact 
    and in particular preserves kernels.  It follows that the $p$-divisible group of the torus $T_{\m}$
    is naturally identified with $\T_{\m}$ (as an {\em fppf}-abelian sheaf, say).
    Multiplication by $p^n$ is {\em fppf}~surjective on any torus, so the snake lemma shows that the canonical
    map $\Gscr_{X,\m}[p^n]\rightarrow \Gscr_X[p^n]$ is surjective with kernel $\T_{\m}[p^n]$ for all $n$, which completes the proof.
\end{proof}



\subsection{Dieudonn\'e modules and de Rham cohomology}

We begin with a brief recall of Dieudonn\'e theory for finite $k$-group schemes of $p$-power order
and $p$-divisible groups.  The standard reference for this material is \cite[II--III]{Fontaine}.
For a modern synopsis, we recommend \cite[\S1.4]{CCO}.
Throughout, we will simply write {\em finite $k$-group} for 
finite, commutative $k$-group scheme of $p$-power order.

As $k$ is {\em perfect}, any finite $k$-group or $p$-divisible group $G$ admits a functorial decomposition   
\begin{equation}\label{Gdecomp}
        G = G^{\et} \times G^{\mult} \times G^{\ll}
\end{equation}
with $G^{\et}$ \'etale (reduced with connected dual),
$G^{\mult}$ multiplicative (connected with reduced dual),
and $G^{\ll}$ local--local (connected with connected dual).
If $G$ is a finite $k$-group, 
then $G=\Spec R$ for some $k$-algebra $R$ that is finite dimensional as 
a $k$-vector space, and the {\em order} of $G$
is $|G| = \dim_k G;$
this recovers the usual notion when $G$ is the constant group scheme
associated to a finite abelian group.
If $G=\{G[p^n]\}_n$ is a $p$-divisibe group,
then there is an integer $h$ so that $|G[p^n]|=p^{nh}$ for all $n$,
and we say that $G$ is of {\em height} $h$, and write $\Ht(G):=h$.

The {\em Dieudonn\'e ring} relative to $k$ is the (associative but noncommutative if $k\neq \FF_p$) $W$-algebra
$\calD_k:=W[F,V]$ generated by variables $F$ (Frobenius) and $V$ (Verscheibung) subject to the relations
\begin{itemize}
    \item $FV = VF = p$
    \item $F\lambda = \sigma(\lambda) F$
    \item $V\lambda = \sigma^{-1}(\lambda) V$
\end{itemize}
for all $\lambda\in W$.  By definition, a {\em Dieudonn\'e module}
is a (left) $\calD_k$-module that is finitely generated as a $W$-module;
equivalently, it is a finite type $W$-module $M$ with additive maps $F,V:M\rightarrow M$
satisfying the three conditions above.  These form an abelian category,
with morphisms $\calD_k$-module homomorphisms, {\em i.e.}~$W$-linear maps that are compatible
with $F$ and $V$.  This category is equipped with an involutive duality functor:
if $M$ is a Dieudonn\'e module of finite $W$-length, 
the {\em dual} of $M$ is the Pontryagin dual $M^{\vee}:=\Hom_W(M,W[1/p]/W)$ as a $W$-module,
with $F_{M^{\vee}}(f)(x):=\sigma f(Vx)$ and $V_{M^{\vee}}(f)(x):=\sigma^{-1}f(Fx)$
for $f\in M^{\vee}$ and $x\in M$.
If $M$ is finite and free as a $W$-module, we set $M^{\vee}:=\Hom_W(M,W)$
with Frobenius and Verscheibung defined in the same way.


\begin{thm}[Main Theorem of Dieudonn\'e Theory]\label{thm:Dieudonne}
    There is an exact anti-equivalence of categories $G\rightsquigarrow \D_k(G)$
    from the category of finite commutative $k$-group schemes of $p$-power order
    to the category of Dieudonn\'e modules of finite $W$-length.  
    If $G=\{G[p^n]\}_n$ is a $p$-divisible group, defining $\D_k(G):=\varprojlim_n \D_k(G[p^n])$
    induces an exact antiequivalence between the category of $p$-divisibe groups and the category of Dieudonn\'e 
    modules that are free of finite rank over $W$.
    Furthermore:
    \begin{enumerate}
        
        \item $\D_k$ is compatible with duality: there is a natural isomorphism of covariant functors
        \begin{equation*}
            \D_k((\cdot)^{\vee}) \simeq \D_k(\cdot)^{\vee}
        \end{equation*}
        
        \item $\D_k$ is compatible with base change: if $k\rightarrow k'$ is any extension of perfect fields,
        there is a natural isomorphism of contravariant functors
        \begin{equation*}
            \D_k(\cdot)\otimes_W W(k') \simeq \D_{k'}((\cdot)\times_k {k'}).
        \end{equation*}
        
        \item The linearizations of $F$ and $V$ on $\D(G)$ correspond via functoriality to 
        the relative Frobenius $G\rightarrow G^{(p)}$ and Verscheibung $V: G^{(p)}\rightarrow G$
        morphisms via compatibility of $\D(\cdot)$ with the base change $\sigma: k\rightarrow k$.
        
        \item For any $G$
            \begin{enumerate}
                \item $G=G^{\et}$ if and only if $F$ is bijective and $V$ is topologically nilpotent on $\D(G)$.
                \item $G=G^{\mult}$ if and only if $F$ is topologically nilpotent and $V$ is bijective on $\D(G)$.
                \item $G=G^{\ll}$ if and only if $F$ and $V$ are both topologically nilpotent on $\D(G)$.
            \end{enumerate}
            
        \item\label{DieudonneDiff} Let $e\in G(k)$ be the identity section. There is a functorial isomorphism of $k$-vector spaces
        \begin{equation*}
            \D(G)/F\D(G) \simeq \omega_G
        \end{equation*}
        where $\omega_G:=H^0(\Spec k, e^*\Omega^1_{G/k})$ is the cotangent space of $G$ at the identity.
        
        \item If $G$ is finite, $\log_p |G| = \length_W \D(G)$, and if $G$ is $p$-divisible,
         $\Ht(G) = \rank_W \D(G)$.\label{lenord}
    
    \end{enumerate}
\end{thm}

\begin{remark}
    When $k$ is clear from context, we will write simply $\D$ in place of $\D_k$.
\end{remark}

A crucial ingredient in the proofs of our main theorems is {\em Oda's theorem} \cite[Corollary 5.11]{Oda},
which provides an explicit description of Dieudonn\'e modules 
of $p$-torsion group schemes of abelian varieties in terms of de~Rham cohomology.
We will only need this result for Jacobians of curves, where it can be stated in terms of the 
de Rham cohomology of the curves themselves.  

Associated to any smooth and proper
curve $X$ over $k$ is the short exact ``Hodge filtration''
sequence of its de~Rham cohomology:
\begin{equation}
    \xymatrix{
    0 \ar[r] & {H^0(X,\Omega^1_{X/k})} \ar[r] & {H^1_{\dR}(X/k)} \ar[r] & {H^1(X,\O_X)} \ar[r] & 0.
    }\label{CrvdRseq}
\end{equation}
This sequence is functorial in finite morphisms of smooth and proper curves $\pi:Y\rightarrow X$
in {\em two ways}: we write $\pi^*$ for the contravariant {\em pullback} map associated to $\pi$,
and $\pi_*$ for the covariant {\em trace} map; we have $\pi_*\pi^*=\deg(\pi)$
and when $Y\rightarrow X$ is generically Galois with group $G$, we moreover have $\pi^*\pi_* = \sum_{g\in G} g^*$.
Cup product pairing induces a canonical autoduality pairing $\langle \cdot,\cdot\rangle_X$ on \eqref{CrvdRseq}, via which $\pi_*$ and $\pi^*$ are adjoint,
for any finite $\pi:Y\rightarrow X$.  In particular, as the absolute Frobenius morphism $\Fr_X:X\rightarrow X$
is finite, the exact sequence \ref{CrvdRseq} is equipped with semilinear endomorphisms
$F := \Fr_X^*$ and $V:=(\Fr_X)_*$ which are (semilinearly) adjoint under $\langle \cdot,\cdot\rangle_X$,
and satisfy $VF=\deg \Fr_X=0$.  The map $V$ is called the {\em Cartier operator}, and admits
a more explicit description, as in \cite[Definition 5.5]{Oda}; see \cite[\S2.1]{CaisHida1}
and {\em cf.}~\cite[\S10 Proposition 9]{SerreTop} for a proof that our definition of $V$ coincides with Oda's.
In this way, \eqref{CrvdRseq} becomes a short exact sequence of Dieudonn\'e modules;
note that $F=0$ on $H^0(X,\Omega^1_{X/k})$ as pullback of any differential form by the $p$-power map
in characteristic $p$ is zero, and by duality $V=0$ in $H^1(X,\O_X)$.

On the other hand, writing $J_X$ for the Jacobian of $X$, the
the relative Verscheibung of $J_X[p]^{(p^{-1})}$ sits in a short exact sequence of finite group schemes
\begin{equation}
    \xymatrix{
        0 \ar[r] & J_X[V] \ar[r] & J_X[p] \ar[r]^-{V} & J_X[F]^{(p^{-1})} \ar[r] & 0
    }.\label{Atorseq}
\end{equation}
This short exact sequence is contravariantly functorial in finite morphisms of smooth and proper curves 
$\pi:Y\rightarrow X$ via Picard functoriality $\pi^*:=\Pic^0(\pi)$ (pullback of line bundles)
and covariantly functorial via Albanese functoriality $\pi_*:=\Alb(\pi)$ (Norm of line bundles).
Applying the Dieudonn\'e functor to \eqref{Atorseq} yields a short exact sequence of Dieudonn\'e
modules
\begin{equation}
    \xymatrix{
            0 \ar[r] & {k\tens_{k,\sigma^{-1}} \D(J_X[F])} \ar[r] & {\D(J_X[p])} \ar[r] &
            {\D(J_X[F])} \ar[r] & 0 \\
    }\label{Dmodseq}
\end{equation}
that is {\em covariantly} functorial in finite morphisms $\pi:Y\rightarrow X$
via $\D(\pi^*)$ and {\em contravariantly} functorial via $\D(\pi_*)$.
The principal polarization of $J_X$ yields a canonical isomorphism $J_X[p]^{\vee}\simeq J_X[p]$
intertwining $(\pi^*)^{\vee}$ with $\pi_*$ and $(\pi_*)^{\vee}$ with $\pi^*$;
via the compatibility of $\D(\cdot)$ with duality, the dual of \eqref{Dmodseq}
is then functorially identified with \eqref{Dmodseq}.

\begin{thm}[Oda]\label{thm:Oda}
    Let $X$ be a smooth and proper curve over $k$ with Jacobian $J_X$.  There is a canonical
    isomorphism of short exact sequences of Dieudonn\'e modules
    \begin{equation*}
        \xymatrix{
            0 \ar[r] & {k\tens_{k,\sigma^{-1}} \D(J_X[F])} \ar[r]\ar[d]^-{\simeq} & {\D(J_X[p])} \ar[r]\ar[d]^-{\simeq}  &
            {\D(J_X[F])} \ar[r]\ar[d]^-{\simeq}  & 0 \\
            0 \ar[r] & {H^0(X,\Omega^1_{X/k})} \ar[r] & {H^1_{\dR}(X/k)} \ar[r] & {H^1(X,\O_X)} \ar[r] & 0
        }
    \end{equation*}
    wherein the top row is \eqref{Dmodseq} and the bottom row is \eqref{CrvdRseq}.
    This isomorphism is compatible with the autodualities of each row, 
    and for any morphism $\pi:Y\rightarrow X$ of smooth and proper curves, intertwines 
    $\D(\pi_*)$ on Dieudonn\'e modules with $\pi^*$ on de~Rham cohomology, and $\D(\pi^*)$ with $\pi_*$.
\end{thm}

\begin{proof}
    This follows from \cite[Corollary 5.11]{Oda} and \cite[Proposition 5.4]{CaisNeron}
\end{proof}

We will also need the following variant of Oda's theorem for reduced but possibly non-smooth curves.
If $X$ is any reduced proper curve over $k$, its Jacobian $J_X:=\Pic^0_{X/k}$
is a smooth commutative group scheme over $k$.  While the kernel of multiplication by $p$
on $J_X$ may {\em not} be finite, the kernel of relative Frobenius $J_X\rightarrow J_X^{(p)}$
{\em is} finite, so we may consider its Dieudonn\'e module.  On the other hand, 
writing $\omega_{X/k}$ for the relative dualizing sheaf of $X$ as in \cite[\S5.2]{GDBC},
Grothendieck's theory of the trace map attached to any finite map of proper reduced curves 
$\pi:Y\rightarrow X$ yields a map on dualizing sheaves $\pi_* : \pi_*\omega_{X/k}\rightarrow \omega_{Y/k}$
that is dual to the pullback map $\pi^*: \O_Y\rightarrow \pi_*\O_X$ via Grothendieck duality; 
see \cite[Appendix A]{CaisHida1}.  If $\pi$ is in addition {\em flat}, there is
$\pi_*\O_Y$ is a finite locally free $\O_X$-module, so one has a trace mapping $\pi_*:\pi_*\O_Y\rightarrow \O_X$
on functions, whose dual is a pullback map $\pi^*:\omega_{X/k}\rightarrow \pi_*\omega_{Y/k}$
on dualizing sheaves that coincides with pullback of differential forms when $X$ and $Y$
are smooth.
In particular, the finite absolute Frobenius map induces semilinear mapping 
\begin{equation*}
    \xymatrix{
        {V: = (\Fr_X)_*:H^0(X,\omega_{X/k})}\ar[r] & { H^0(X,\omega_{X/k})}
        }
\end{equation*}
which coincides with the Cartier operator when $X$ is smooth.  In this way, 
$H^0(X,\omega_{X/k})$ is naturally a Dieudonn\'e module with $F=0$.

\begin{lem}\label{SingularOda}
    Let $X$ be a proper curve over $k$, and $J_X:=\Pic^0_{X/k}$ its Jacobian.  There is a natural isomorphism of Dieudonn\'e modules
    \begin{equation}
        \D(J_X[F]) \simeq H^0(X,\omega_{X/k})\label{genisom}
    \end{equation}
    If $\pi:Y\rightarrow X$ is a finite map of reduced curves, then \eqref{genisom} intertwines
    $\D(\pi^*)$ with $\pi_*$, and if $\pi$ is in addition flat, it intertwines $\D(\pi_*)$ with $\pi^*$.
\end{lem}

\begin{proof}
    Let $G:=J_X[F]$.  By Theorem \ref{thm:Dieudonne} \ref{DieudonneDiff},
    we have a natural isomorphism of Dieudonn\'e modules $\D(G)\simeq \omega_G$.
    Now the contravariant cotangent space functor is right exact, so as $F:J_X\rightarrow J_X^{(p)}$
    induces the zero map on cotangent spaces, we conclude that
    the closed immersion $G\hookrightarrow J_X$ induces a functorial isomorphism $\omega_G\simeq \omega_{J_X}$.
    It therefore suffices to prove that there is a canonical isomorphism 
    $H^0(X,\omega_{X/k})\simeq \omega_{J_X}$
    intertwining $\pi_*$ with $\Pic(\pi)^*$ and $\pi_*$ with $\Alb(\pi)^*$ when $\pi$ is flat.
    To verify this, we may dualize, where by Grothendieck duality we seek an isomorphism 
    $H^1(X,\O_X) \simeq \Lie(J_X)$ intertwining $\pi^*$ with $\Lie(\Pic(\pi))$
    and, when $\pi$ is flat, $\pi_*$ with $\Lie(\Alb(\pi))$.
    The required isomorphism is provided by \cite[Proposition 1.3]{LiuLorenzini}.
\end{proof}


\begin{remark}
    In fact, using Lemma \ref{SingularOda} as a starting point, it is possible to generalize
    Oda's Theorem \ref{thm:Oda} to the case of {\em arbitrary} proper curves.
\end{remark}

Besides Oda's theorem, there is one other crucial tool in our proofs, Nakajima's {\em equivariant Deuuring Shafrevich formula}
\cite{Nakajima}.  To state it in the form we will need, we first note that any Dieudonn\'e module $M$
admits a functorial decomposition
\begin{equation*}
    M = M^{\Vbij} \oplus M^{\Fbij} \oplus M^{\FVnil}
\end{equation*}
where $M^{\Vbij}$ (respecpectively $M^{\Fbij}$) is the maximal $\calD_k$-stable submodule
on which $V$ (resp. $F$) is bijective and $F$ (resp. $V$) is topologically nilpotent
and $M^{\FVnil}$ is the maximal submodule on which both $F$ and $V$ are topologically nilpotent.
While the existence of this decomposition is a fact of ``pure''
semilinear algebra, it corresponds to \eqref{Gdecomp} via the equivalence of Theorem \ref{thm:Oda}. 

\begin{thm}[Nakajima]\label{Nakajima}
    Let $\pi:Y\rightarrow X$ be a finite and generically Galois map of 
    smooth proper and geometrically connected curves over an algebraically closed field $k$
    of characteristic $p>0$, with group $G$ that is a $p$-group.
    Let $\Sigma \subset X(k)$ be the set of branch points of $\pi$,
    and $S$ any finite and nonempty set of points of $X$ containing $\Sigma$,
    and write $g_X$ and $\gamma_X$ for the genus and $p$-rank of $X$, respectively.
    \begin{enumerate}
            \item The $k[G]$-modules $H^0(Y,\Omega^1_{Y/k}(\pi^{-1}S))^{\Vbij}$
            and $H^1(X,\O_X(-\pi^{-1}S))^{\Fbij}$ are each free of rank $\gamma_X+|S| - 1$.
            \label{NakBij}
            
            \item If $\Sigma=\emptyset$, so $\pi:Y\rightarrow X$ is \'etale, then $H^0(Y,\Omega^1_{Y/k})^{\Vnil}$
            and $H^1(Y,\O_Y)^{\Fnil}$ are each free $k[G]$-modules of rank $g-\gamma$.\label{NakNil}
    \end{enumerate}
\end{thm}

\begin{proof}
    Serre duality restricts to a functorial and perfect duality pairing 
    between the two spaces in each of \eqref{NakBij} and \eqref{NakNil},
    so it suffices to prove that one of them is free of the asserted rank.
    Working with spaces of differentials, \eqref{NakBij} then follows
    immediately from \cite[Theorem 1]{Nakajima}, while \eqref{NakNil}
    follows from \cite[Theorem 3]{Nakajima}.
\end{proof}

In the remainder of this section, we record some additional results on the
behavior of the trace mapping attached to finite maps of curves.

\begin{prop}\label{PicInj}
    Let $\pi: Y\rightarrow X$ be a finite, flat and generically \'etale map 
    of proper and geometrically irreducible curves over a field $k$.  If there is a geometric point of $X$
    over which $\pi$ is totally ramified, then $\pi^*: \Pic^0_{X/k} \rightarrow \Pic^0_{Y/k}$
    has trivial scheme-theoretic kernel. 
\end{prop}

\begin{proof}
    The proof of \cite[Lemma 2.2.2]{CaisHida1} goes through verbatim in the present (more general)
    situation, noting that the hypothesis that $X$ and $Y$ be smooth in {\em loc.~cit.}~is used
    {\em only} to ensure that $\pi:Y\rightarrow X$ is flat.
\end{proof}

\begin{prop}\label{PicInjCor}
    Let $\pi:Y\rightarrow X$ be a finite \'etale map of proper, smooth, and geometrically connected curves
    over $k$.  Let $S\subseteq X(\overline{k})$ be a finite and nonempty set of geometric points of $X$ and $T:=\pi^{-1}S$, 
    and consider $S$ and $T$
    as (reduced and effective) divisors on $X$ and $Y$, respectively.  The trace mapping
    \begin{equation}
        \xymatrix{
            {\pi_* : H^0(\Omega^1_{Y/k}(T))} \ar[r] &  {H^0(\Omega^1_{X/k}(S))}
        }\label{trmap}
    \end{equation}
    is surjective.
\end{prop}

\begin{proof}
    As the formation of \eqref{trmap} is compatible with base change, by passing to a finite extension of $k$ if need be,
    we may assume that $S\subseteq X(k)$, and we will induct on $|S|$.  Suppose first that $S=\{s\}$ is a single point,
    and denote by $Y_T$ the singular curve associated to the modulus $T$.  Since $T$ is reduced and $\pi(T) = s\in X(k)$,
    the restriction of $\pi: Y \rightarrow X $ to $T$ scheme-theoretically factors through $X(k)$ \cite[\href{https://stacks.math.columbia.edu/tag/0356}{Tag 0356}]{stacks-project}. 
    By the universal property of $Y\rightarrow Y_T$, 
    we obtain a unique map 
    $\rho: Y_T\rightarrow X$.  
    through which $\pi$ factors.  Now the curve $Y_T$ is by construction geometrically integral, hence (geometrically) reduced
    and therefore Cohen-Macaulay, and $X$ is by hypothesis smooth.
    It follows from the Corollary to Theorem 23.1 in \cite{Matsumura} that $\rho$ is flat, 
    and necessarily
    totally ramified over $s$.  By Proposition \ref{PicInj},
    the map $\rho^*: \Pic^0_{X/k} \rightarrow \Pic^0_{Y_T/k}$ has trivial scheme-theoretic kernel.
    using the fact that $\Lie(\cdot)$ is left exact, together with \cite[Proposition 1.3]{LiuLorenzini},
    we deduce that $\rho^*: H^1(X,\O_X)\rightarrow H^1(Y_T,\O_{Y_T})$ is injective,
    so by Grothendieck duality the trace map 
    $$\rho_*: H^0(Y_T,\omega_{Y_T})\rightarrow H^0(X,\Omega^1_X)= H^0(X,\Omega^1_X(S))$$
    is surjective; the final equality uses the fact that the sum of the residues at all singular points of a meromorphic differential
    on a proper curve is zero.  
    On the other hand, by Lemma \ref{lem:dualizingdesc},
    the normalization map $\nu:Y\rightarrow Y_T$  
    induces an identification $\nu_*: H^0(Y,\Omega^1_Y(T))\simeq H^0(Y_T,\omega_{Y_T/k})$.
    As $\pi = \rho \circ \nu$, we conclude that $\pi_*$ is surjective when $S=\{s\}$.
    
    Now supposing that the result holds for a given finite and nonempty set $S\subseteq X(k)$, let $s\in X(k)$ be a point not
    in $S$, and set $S':=S\cup \{s\}$ and $T':=T \cup \pi^{-1}s$.  We then have a commutative diagram 
    of short exact sequences
    \begin{equation}
        \xymatrix@C=35pt{
            0 \ar[r] & {H^0(\Omega^1_{Y}(T))} \ar[d]^-{\pi_*} \ar[r] & {H^0(\Omega^1_{Y}(T'))} \ar[d]^-{\pi_*} \ar[r]^-{\eta \mapsto (\res_t\eta)_t} & {\displaystyle\bigoplus_{t \in \pi^{-1}(s)} k} \ar[r]\ar[d]^-{\sum_{t\in \pi^{-1}(s)}} & 0 \\
            0 \ar[r] & {H^0(\Omega^1_{Y}(S))}  \ar[r] & {H^0(\Omega^1_{X}(S'))}  \ar[r]_-{\eta \mapsto \res_s\eta} & {k} \ar[r] & 0 
        }\label{indstepdia}
    \end{equation}
    Indeed, exactness of the rows is a straightforward consequence of the Riemann--Roch theorem,
    and the (right square of the) diagram commutes thanks to \cite[Remark 2.2]{CaisHida1}. 
    Now the left vertical arrow of \eqref{indstepdia} is surjective by our inductive hypothesis, while surjectivity of the right vertical arrow is clear.  We conclude that the middle vertical arrow is surjective as well.
\end{proof}

\section{Proofs}

In this section, we 
prove our main results stated in \S\ref{intro},
whose setting and notation we maintain, and now briefly recall.
Throughout, we will fix a function field $K$ in one variable over $k$
and a $\Gamma$-extension $L/K$ with $k$ algebraically closed in $L$ 
and $\Gamma$ an infinite pro-$p$ group
equipped with a countable basis of the identity $\Gamma=\Gamma_0\supsetneq \Gamma_1\supsetneq\cdots$ consisting of open, normal subgroups.
We write $K_n:=L^{\Gamma_n}$ for the fixed field of $\Gamma_n$, and $X_n$
for the unique smooth, projective and geometrically connected curve over $k$ with function field $K_n$,
and for all $n\ge m$, we denote by $\pi_{n,m}:X_n\rightarrow X_m$ the map of curves corresponding to 
the inclusion of function fields $K_m\hookrightarrow K_n$.
For all $n$, we assume that $\pi_{n,0}: X_n\rightarrow X_0$ is \'etale 
outside a finite (possibly empty) set $\Sigma$ of closed points of $X_0$ that is independent of $n$,
and totally ramified at every point of $\Sigma$.

\begin{remark}\label{kcountable}
    Let $U$ be the complement of $\Sigma$ in $X_0$, and $\gamma$ the $p$-rank of $X_0$.
    By \cite{Shafarevich}, the maximal pro-$p$ quotient
    $\pi_1^{\et}(U_{\o{k}})^{(p)}$ of the geometric \'etale fundamental group of $U$ is
    a {\em free} pro-$p$ group 
     on 
    $\gamma$ generators when $\Sigma=\emptyset$, and on as many generators as the cardinality of $\o{k}$
    when $\Sigma\neq\emptyset$; see \cite[\S1]{Gille} for a modern proof of this fact
    via cohomological dimension and \'etale cohomology.  
    It follows from this and \cite[Corollary 2.6.6]{RZ},
    that $\pi_1^{\et}(U)^{(p)}$ admits a countable basis of the identity consisting of open normal subgroups 
    if and only if $k$ is countable; in particular, our assumption that $\Gamma$
    admit a countable basis of the identity is automatically verified  
    whenever $k$ itself is countable.     
\end{remark}

When $\Sigma\neq\emptyset$, 
we set $S:=\Sigma$, and when $\Sigma=\emptyset$, 
we fix a closed point $x_0$ of $X_0$ and put $S:=\{x_0\}$.
For each $n$, let $S_n:=(\pi_{n,0}^{-1}S)_{\red}$ be the reduced closed subscheme
underlying the scheme-theoretic fiber of $S$ in $X_n$, so $S_0=S$. In the \'etale case
$\Sigma=\emptyset$, note that $S_n$ is just the (scheme-theoretic) fiber over $x_0$,
while in the ramified case $\Sigma\neq\emptyset$, it is the reduced closed subscheme
of $X_n$ at which $\pi_{n,0}$ is non-\'etale.  For each $n$,
we write $\Gscr_{n}:=J_{X_n}[p^{\infty}]$ for the $p$-divisible group of the Jacobian 
of $X_n$, and $\Gscr_{n,S_n}:=J_{X_n,S_n}[p^n]$ for inductive system of $p^n$-torsion
group schemes on the generalized Jacobian with modulus $S_n$ (Definition \ref{GenpdivDef}). 
By Corollary \ref{GJstr}, each $\Gscr_{n,S_n}$ is a $p$-divisible group, and there is a
canonical exact sequence of $p$-divisible groups
\begin{equation}
            \xymatrix{
                0 \ar[r] & {\T_n} 
                \ar[r] & {\Gscr_{n,S}} \ar[r] & {\Gscr_n} \ar[r] & 0
            }\label{eq:basicses}
\end{equation} 
with $\T_n$ of multiplicative type.  
We write $\Qscr_n:=\T_n^{\vee}$ for the dual $p$-divisible group, which is \'etale
as $\T_n$ is multiplicative. 
It follows from Proposition \ref{GJfunc} that $\Gamma$ acts on $J_{X_n,S_n}$ in two ways,
with $\gamma\in \Gamma$ acting either as $\gamma^*$, or as $\gamma_*$.  We will refer to the first
of these actions as the Picard, or $(\cdot)^*$-action, and the second as the Albanese
or $(\cdot)_*$-action.  Note that as $\gamma_*\gamma^* = \deg(\gamma)=1$, it follows that $\gamma_*=(\gamma^*)^{-1}$ for all $\gamma\in \Gamma$.
These actions induce actions of $\Gamma$ on \eqref{eq:basicses},
with $\Gamma_n$ acting trivially.
For each pair of integers $n> m$, the degeneracy maps $\pi_{n,m}:X_n\rightarrow X_m$
likewise induce {\em two} morphisms of the exact sequence \eqref{eq:basicses} by Proposition \ref{GJfunc},
and we will frequently abbreviate $\pi^*:=\pi_{n,m}^*$ for the morphism
induced by Picard functoriality, and $\pi_*:=(\pi_{n,m})_*$ for that induced
by Albanese functoriality; note that $\pi^*$ is equivariant with respect to the $(\cdot)^*$-action
of $\Gamma$, while $\pi_*$ is equivariant with respect to the $(\cdot)_*$-action.
Furthermore, the autodiality of Jacobians induces a canonical identification $\Gscr_n^{\vee}\simeq \Gscr_n$
for all $n$, via which one has $(\pi_*)^{\vee} = \pi^*$ as morphisms $\Gscr_m\rightarrow \Gscr_n$.

\begin{defn}\label{def:limitmod}    
    Let $\star\in \{\et,\mult,\ll,\null\}$.
    For each $n$, set $\D_n^{\star}:=\D(\Gscr_n^{\star})$ and write
    $\rho:=\D(\pi^*)$ 
    for the induced transition maps.  When $\star = \mult,\ll$, we likewise
    put $\D^{\star}_{n,S_n}:=\D(\Gscr_{n,S_n}^{\star})$, equipped with 
    transition maps $\rho=\D(\pi^*)$.  For $\star=\et$, we instead define
    \begin{equation*}
        \D_{n,S_n}^{\et}:=\D((\Gscr_{n,S_n}^{\mult})^{\vee}) = \D((\Gscr_{n,S_n}^{\vee})^{\et})
    \end{equation*}
    which we make into an inverse system via the transition maps $\rho':=\D((\pi_*)^{\vee})$.
    We equip $\D_n^{\star}$ and $\D_{n,S}^{\mult}$ with the 
    {\em left} action of $\Gamma$ in which
    $\gamma\in \Gamma$ acts as $\D(\gamma^*)$, while we equip $\D_{n,S}^{\et}$ with the {\em left} $\Gamma$-action
    wherein $\gamma$ acts as $\D((\gamma_*)^{\vee})$.
    These actions commutes with $F$ and $V$, and $\Gamma_n$ acts trivially.

    For any $\star$, we then define the Iwasawa--Dieudonn\'e modules
    \begin{equation*}
        \D^{\star}:=\varprojlim_{n,\rho} \D_n^{\star}
    \end{equation*}
    with projective limits taken with respect to the indicated transition maps.
    We similarly define
    \begin{equation*}
        \D_S^{\star}:=\varprojlim_{n,\rho} \D_{n,S_n}^{\star} \quad\text{for}\ \star\in \{\mult,\ll\}, 
        \quad\text{and}\quad \D_S^{\et}:=\varprojlim_{n,\rho'} \D_{n,S_n}^{\et}
    \end{equation*}
    
\end{defn}

In each case, the given transition maps are equivariant with respect to the specified 
(left) actions of $\Gamma$, so $\D^{\star}$ and $\D^{\star}_S$ are
naturally (left) $\Lambda$-modules with
additive maps $F,V$ that are semilinear over 
$\sigma$ and $\sigma^{-1}$, respectively, and satisfy $FV=VF=p$.

\begin{remark}\label{rem:rho}
    The definition of $\D_{n,S_n}^{\et}=\D((\Gscr_{n,S_n}^{\mult})^{\vee})$ may seem
    strange at first, but notice that as $\Gscr_{n,S_n}$ is an extension of $\Gscr_n$ by a $p$-divisible group
    of {\em multiplicative} type,
    we in fact have $\Gscr_{n,S_n}^{\et} = \Gscr_{n}^{\et}$, so via the autoduality identification 
    $\Gscr_n^{\et} \simeq (\Gscr_n^{\mult})^{\vee}$ we see that $\Gscr_n^{\et}$ is a sub $p$-divisible group of 
    $(\Gscr_{n,S_n}^{\mult})^{\vee}$, with \'etale quotient $\Qscr_n:=\T_n^{\vee}$.  Note that 
    $\Gscr_{n,S_n}$ is {\em not} auto-dual; consequently, its \'etale part is too small
    to effectively study $\Gscr_n^{\et}$. Via
    the identifications $\Gscr_n^{\vee}\simeq \Gscr_n$,
    the transition maps $\rho:=\D(\pi^*)$ and $\rho':=\D((\pi_*)^{\vee})$
    {\em coincide} as maps $\D(\Gscr_n) \rightarrow \D(\Gscr_m)$, so that the projective limit $\D^{\et}$ can be formed via either.
\end{remark}

\begin{remark}\label{loclocS}
    As in Remark \ref{rem:rho}, the kernel of $\Gscr_{n,S_n}\twoheadrightarrow \Gscr_n$
    is of multiplicative type, whence we have an isomorphism $\Gscr_{n,S}^{\ll}\simeq \Gscr_{n}^{\ll}$
    for all $n$, so in particular $\D^{\ll} \simeq \D^{\ll}_S$.
\end{remark}

For $\star\in \{\mult,\ll\}$, the compatibility of $\D(\cdot)$ with duality
gives a natural isomorphisms of Dieudonn\'e modules 
\begin{equation}
    \D((\Gscr_{n,S_n}^{\star})^{\vee}) \simeq \Hom_W(\D_{n,S_n}^{\star},W)
    \label{DDual}
\end{equation}
intertwining $\D((\gamma_*^{-1})^{\vee})$ with $\D(\gamma_*^{-1})^{\vee} = \D(\gamma^*)^{\vee}$. 
It follows that the corresponding  $W$-bilinear perfect duality pairing
\begin{equation}
    \xymatrix{
        {\langle\cdot, \cdot\rangle : \D_{n,S_n}^{\star} \times \D((\Gscr_{n,S_n}^{\star})^{\vee}}) \ar[r] & W
        }\label{Dpairing}
\end{equation}
satisfies
\begin{equation*}
    \langle \gamma \cdot x,  y\rangle =
    \langle \D(\gamma^*)(x), y \rangle = \langle x, \D((\gamma_*^{-1})^{\vee})(y) \rangle 
\end{equation*}
For any $n\ge m$, we likewise have
\begin{equation*}
    \langle \D(\pi^*) x , \D((\pi_*)^{\vee}) y \rangle = 
    \langle x , \D((\pi_*)^{\vee})\circ \D((\pi^*)^{\vee})(y) \rangle
    =\langle x , \D((\pi^*)^{\vee}\circ (\pi_*)^{\vee})(y) \rangle 
    = \langle x , \D( (\pi_*\pi^* )^{\vee})(y)\rangle.
\end{equation*}
From Proposition \ref{GJfunc}, we have the identity
\begin{equation*}
    \pi_*\pi^*  = \sum_{\gamma\in \Gamma_m/\Gamma_n} \gamma^* = \sum_{\gamma\in \Gamma_m/\Gamma_n} (\gamma_*)^{-1},
\end{equation*}
Using additivity of the functors $\D(\cdot)$ and $(\cdot)^{\vee}$
and the fact that the identification $(\Gscr_n^{\ll})^{\vee}\simeq \Gscr_n^{\ll}$,
intertwines $(\gamma_*)^{\vee}$ with $\gamma^*$,
we deduce that the duality pairing
between $\D_{n,S_n}^{\mult}$ and $\D_{n,S_n}^{\et}$
and between $\D_{n,S_n}^{\ll}\simeq \D_n^{\ll}$ and itself
induced by \eqref{Dpairing}
verify the hypotheses of Proposition \ref{prop:NCDual}.

\subsection{The ramified case}

By hypothesis, $\pi_{n,0}:X_n\rightarrow X_0$ is totally ramified at every point of $S_n$,
so induces an identification of \'etale $k$-schemes $S_n\simeq S_0=S$ for all $n$. We will
henceforth make these identifications,
which induce identifications $\T_n\simeq \T_0$ and $\Qscr_n\simeq\Qscr_0$ for all $n$, which we likewise
make, denoting these common $p$-divisible groups simply by $\T_S$ and $\Qscr_S$, respectively.
We similarly often write $\Gscr_{n,S}$ for $\Gscr_{n,S_n}$,
and $\D_{n,S}^{\star}$ in place of $\D_{n,S_n}^{\star}$.

In order to prove Theorem \ref{thmD},
we will apply Propositions \ref{prop:NCfree} and \ref{prop:NCDual}
to the $\Gamma$-towers of $W$-modules $\{\D_{n,S}^{\mult},\rho\}$
and $\{\D_{n,S}^{\et},\rho'\}$.
To do so, it suffices to prove:

\begin{lem}\label{lem:DmapSurj}
    For each pair of positive integers $m\le n$, the maps
    \begin{equation*}
            \xymatrix{
            {\o{\rho} : \o{\D}_{n,S}^{\mult} } \ar[r] & {\o{\D}_{m,S}^{\mult}}
            }  
            \quad\text{and}\quad
            \xymatrix{
            {\o{\rho}' : \o{\D}_{n,S}^{\et}} \ar[r] & {\o{\D}_{m,S}^{\et}}
            }
    \end{equation*}
    are surjective.  Moreover, $\o{\D}_{n,S}^{\mult}$ and $\o{\D}_{n,S}^{\et}$
    are free $\Omega_n$-modules of rank $d$ for all $n$.
\end{lem}

\begin{proof}
    We may assume that $k$ is algebraically closed.
    Oda's Theorem \ref{thm:Oda} 
    and the isomorphism \eqref{DDual} 
    yield natural isomorphisms of $\Omega_n$-modules 
    \begin{subequations}    
    \begin{equation}
            \o{\D}_{n,S}^{\mult} = \D(\Gscr_{n,S}^{\mult})\otimes_W k \simeq \D(J_{X_n,S}[p]^{\mult}) \simeq 
            H^0(X_n,\Omega^1_{X_n/k}(S))^{\Vbij}\label{Dmultk}
    \end{equation}
    \begin{equation}
            \o{\D}_{n,S}^{\et} = \D((\Gscr_{n,S}^{\mult})^{\vee})\otimes_W k \simeq \D(J_{X_n,S}[p]^{\mult})^{\vee} \simeq 
            \left(H^0(X_n,\Omega^1_{X_n/k}(S))^{\Vbij}\right)^{\vee}\label{Detk}
    \end{equation}
    \end{subequations}
    It follows from this and Nakajima's Theorem \ref{Nakajima} \ref{NakBij}
    that $\o{\D}^{\mult}_{n,S}$ and $\o{\D}^{\et}_{n,S}$
    are each free $\Omega_n$-modules of rank $d$.
    Now \eqref{Dmultk}--\eqref{Detk} are
     compatible with change in $n$, in the sense that
    for all $m\le n$, each of the two diagrams
    \begin{equation*}
        \xymatrix{
            {\o{\D}_{n,S}^{\mult}} \ar[r]^-{\simeq}\ar[d]_-{\o{\rho}} & {H^0(X_n,\Omega^1_{X_n/k}(S))^{\Vbij}}\ar[d]^-{\pi_*} \\
            {\o{\D}_{m,S}^{\mult}} \ar[r]^-{\simeq} & {H^0(X_m,\Omega^1_{X_m/k}(S))^{\Vbij}}
        }
        \xymatrix{
            {\o{\D}_{n,S}^{\et} } \ar[d]_-{\o{\rho}'} & \ar[l]_-{\simeq} {\left(H^0(X_n,\Omega^1_{X_n/k}(S))^{\Vbij}\right)^{\vee}}\ar[d]^-{(\pi^*)^{\vee}} \\
            {\o{\D}_{m,S}^{\et} } & \ar[l]_-{\simeq} {\left(H^0(X_m,\Omega^1_{X_m/k}(S))^{\Vbij}\right)^{\vee}}
        }
    \end{equation*}
    commute.  
    We are thereby reduced to the claim that 
    the map $\pi^*$ (respectively $\pi_*$) is injective (resp. surjective) on differential forms. 
    Pullback of differential forms is injective as $\pi_{n,m}$ is generically \'etale, 
    and to see that trace is surjective we may dualize, apply Serre duality, and argue that the
    pullback map
    \begin{equation}
        \xymatrix{ 
            {\pi^*: H^1(X_m,\O_{X_m}(-S))}\ar[r] & {H^1(X_n,\O_{X_n}(-S))}
            }\label{H1polesPB}
    \end{equation}
    is injective.  This pullback map fits into a commutative diagram with exact rows
    \begin{equation*}
        \xymatrix{
            0 \ar[r] & {H^0(X_m,\O_{X_m})} \ar[r]\ar[d] & {H^0(X_m, \O_S)} \ar[r]\ar[d] & H^1(X_m, \O_{X_m}(-S)) \ar[r]\ar[d]^-{\pi^*} & H^1(X_m,\O_{X_m}) \ar[r]\ar[d]^-{\pi^*} & 0 \\
            0 \ar[r] & {H^0(X_n,\O_{X_n})} \ar[r] & {H^0(X_n, \O_S)} \ar[r] & H^1(X_n, \O_{X_n}(-S)) \ar[r] & H^1(X_n,\O_{X_n}) \ar[r] & 0
        }
    \end{equation*}
    Since $X_n$ is (geometrically) connected for all $n$, we have $H^0(X_n,\O_{X_n})\simeq k$ and the first 
    vertical map is an isomorphism.  As $S$ is reduced, we have $H^0(X_n, \O_S) \simeq k^{\deg S}$ 
    for all $n$, and the second vertical map is likewise an isomorphism.  
    The injectivity of the final vertical map follows from
    Proposition \ref{PicInj}, arguing as in the proof of Proposition \ref{PicInjCor}. 
    A diagram chase then
    implies that \eqref{H1polesPB} is injective as well.  
\end{proof}

\begin{proof}[Proof of Theorem \ref{thmD}]
    Thanks to Lemma \ref{lem:DmapSurj} and our calculations immediately below \eqref{Dpairing},
    the $\Gamma$-towers of $W$-modules $\{\D_{n,S}^{\mult},\rho\}$
    and $\{\D_{n,S}^{\et},\rho'\}$ satisfy the
    hypotheses of Propositions \ref{prop:NCfree} and \ref{prop:NCDual},
    whose conclusions immediately give Theorem \ref{thmD}.    
\end{proof}

\begin{lem}
	For each pair of positive integers $m\le n$, the map $\pi_{n,m}: X_n\rightarrow X_m$
	induces commutative diagrams of $p$-divisible groups
	\begin{subequations}
	\begin{equation}
    		\xymatrix@C=50pt{
        			0 \ar[r] & {\T_S} \ar[r]   & 
			{\Gscr_{n,S}} \ar[r]  & {\Gscr_{n}} 
			\ar[r] & 0 \\
        			0 \ar[r] & {\T_S} \ar[r]\ar[u]^-{\id} & {\Gscr_{m,S}} 
			\ar[r]	\ar[u]^-{\pi^*} & {\Gscr_{m}} 
			\ar[r]\ar[u]^-{\pi^*} & 0
    }\label{eq:Picram}
\end{equation}
and
\begin{equation}
    		\xymatrix@C=50pt{
        			0 \ar[r] & {\T_S} \ar[r]  \ar[d]^-{[\Gamma_m:\Gamma_n]} & 
			{\Gscr_{n,S}} \ar[r] \ar[d]^-{\pi_*} & {\Gscr_{n}} 
			\ar[r] \ar[d]^-{\pi_*} & 0 \\
        			0 \ar[r] & {\T_S} \ar[r] & {\Gscr_{m,S}} 
			\ar[r]	& {\Gscr_{m}} 
			\ar[r] & 0
    }\label{eq:Albram}
\end{equation}
\end{subequations}
	compatibly with change in $m,n$.
	Moreover, \eqref{eq:Pic} and \eqref{eq:Alb} are equivariant for the $(\cdot)^*$-action and the 
	$(\cdot)_*$-action of $\Gamma$, respectively.
\end{lem}

\begin{proof}
    This follows immediately from Proposition \ref{GJfunc} and Corollary \ref{GJstr}.
\end{proof}

Passing to multiplicative parts on \eqref{eq:Picram} and applying the (contravariant) Dieudonn\'e module functor 
yields, for each $m\le n$, the commutative diagram of Dieudonn\'e modules, whose rows are short exact
	\begin{equation}
    		\xymatrix@C=50pt{
        			0 \ar[r] & {\D_n^{\mult}} \ar[r]  \ar[d]^-{\rho} & 
			{\D_{n,S}^{\mult}} \ar[r] \ar[d]^-{\rho} & {\D(\T_S)} 
			\ar[r] \ar[d]^-{\id} & 0 \\
        			0 \ar[r] & {\D_m^{\mult}} \ar[r] & {\D_{m,S}^{\mult}} 
			\ar[r]	& {\D(\T_S)} 
			\ar[r] & 0
    }\label{eq:DPic}
\end{equation}
Similarly, passing to multiplicative parts on \eqref{eq:Albram}, dualizing, and applying $\D(\cdot)$ yields
\begin{equation}
    		\xymatrix@C=50pt{
        			0 \ar[r] & {\D(\Qscr_S)} \ar[r]  \ar[d]^-{[\Gamma_m:\Gamma_n]} & {\D_{n,S}^{\et}} \ar[r] \ar[d]^-{\rho'} & {\D_n^{\et}}  \ar[r] 
        			\ar[d]^-{\rho'=\rho} & 0 \\
        			0 \ar[r] & {\D(\Qscr_S)} \ar[r] & {\D_{m,S}^{\et}} 
			\ar[r]	& {\D_{m}^{\et}} 
			\ar[r] & 0
    }\label{eq:DAlb}
\end{equation}
where we have used the canonical identifications $\Gscr_n^{\vee}\simeq \Gscr_n$
coming from principal polarizations on Jacobians.
As duality of Jacobians interchanges Picard and Albanese functoriality, 
these identifications are
 equivariant 
for the action of $(\gamma_*)^{\vee}$ on $\Gscr_n^{\vee}$
and $\gamma^*$ on $\Gscr_n$, and identify the
transition maps $\rho'$ and $\rho$ on $\D_n^{\et}$ as in Remark \ref{rem:rho}.
These diagrams give short exact sequences of projective systems.

To prove Theorem \ref{thmE}, 
we will need to apply the functor $\varprojlim$
to the exact sequences of inverse systems given by \eqref{eq:DPic}
and \eqref{eq:DAlb}.  To analyze the result,
we will need to know that certain derived limits vanish,
which is the content of the next Lemma:

\begin{lem}\label{lem:Dnsurj}
     For each pair of positive integers $m\le n$, the map
    \begin{equation*}
        \xymatrix{
            {\rho : \D_{n}} \ar[r] & {\D_{m}}
            }
    \end{equation*}
    is surjective.  In particular, the map $\rho: \D_{n}^{\star}\rightarrow \D_{m}^{\star}$
    is surjective for any $\star \in \{\et,\mult, \ll,\null\}$.
\end{lem}

\begin{proof}
    As $\D_n^{\star}$ is a functorial direct summand of $\D_n$, the second assertion
    follows from the first.  To prove the first,
    it suffices to do so after applying $(\cdot)\otimes_W k$,
    and we may assume that $k$ is algebraically closed.
    Arguing as in the proof of Lemma \ref{lem:DmapSurj},
    by Oda's Theorem \ref{thm:Oda}, for all $m\le n$ we have a commutative 
    diagram with indicated isomorphisms
    \begin{equation*}
        \xymatrix{
            {\o{\D}_n  \simeq \D(J_{X_n}[p])} \ar[r]^-{\simeq}\ar[d]_-{\o{\rho}} & {H^1_{\dR}(X_n/k)}\ar[d]^-{\pi_*} \\ 
            {\o{\D}_m  \simeq \D(J_{X_m}[p])} \ar[r]^-{\simeq} & {H^1_{\dR}(X_m/k)}
        }
    \end{equation*}
    so it suffices to prove that the trace map on de Rham cohomology is surjective.
    This map fits into a commutative diagram with short exact rows
    \begin{equation}
        \xymatrix{
            0 \ar[r] & {H^0(X_n,\Omega^1_{X_n/k})} \ar[r]\ar[d]^-{\pi_*} & {H^1_{\dR}(X_n/k)}   \ar[r] \ar[d]^-{\pi_*}& {H^1(X_n,\O_{X_n})} \ar[r]\ar[d]^-{\pi_*} & 0 \\
            0 \ar[r]& {H^0(X_m,\Omega^1_{X_m/k})} \ar[r] & {H^1_{\dR}(X_m/k)}   \ar[r] & {H^1(X_m,\O_{X_m})} \ar[r] & 0
        }
    \end{equation}
    so it is enough to prove that the flanking vertical maps are surjective.
    Dualizing and using Serre duality, we equivalently wish to show that the
    pullback maps
    \begin{equation}
        \xymatrix{
            {\pi^*: H^1(X_m,\O_{X_m})} \ar[r] & {H^1(X_n,\O_{X_n})}
        }\quad\text{and}\quad
        \xymatrix{
            {\pi^*: H^0(X_m,\Omega^1_{X_m/k})} \ar[r] & {H^0(X_n,\Omega^1_{X_n/k})}
        }\label{HdRsurjRed}
    \end{equation}
    are injective.  As in the proof of Lemma \ref{lem:DmapSurj}, the second
    map of \eqref{HdRsurjRed} is injective as $\pi_{n,m}$
    is generically \'etale, while the first map is injective
    due to Proposition \ref{PicInj}, as in the proof of Proposition \ref{PicInjCor}.
\end{proof}

\begin{proof}[Proof of Theorem \ref{thmE}]
    We apply $\varprojlim$ to the exact sequences of inverse systems
    given by \eqref{eq:DPic} and \eqref{eq:DAlb}.
    By Lemma \ref{lem:Dnsurj}, the inverse system $\{\D_n^{\mult}, \rho\}$
    is Mittag--Leffler, so the derived limit $\varprojlim^1_{n,\rho} \D_n^{\mult}$
    vanishes, which yields the short exact sequence of Theorem \ref{thmE} \ref{thmE:str}.
    Likewise, $\D(\Qscr_S)$ is a finite free $W$-module, so it follows
    from Lemma \ref{limitvanishing} 
    that applying $\varprojlim$ to \eqref{eq:DAlb}
    results in an isomorphism of $\Lambda$-modules $\D^{\et}_{S} \simeq \D^{\et}$,
    as in Theorem \ref{thmE} \ref{thmE:str}.  This isomorphism
    and the control aspect of Theorem \ref{thmD}
    yield canonical isomorphisms of
    $\Lambda_n$-modules
    \begin{equation*}
        \Lambda_n \tens_{\Lambda}\D^{\et} \simeq \Lambda_n \tens_{\Lambda} \D_S^{\et} \simeq \D_{n,S}^{\et}
    \end{equation*}
    and the first short exact sequence of Theorem \ref{thmE} \ref{thmE:control}
    follows from \eqref{eq:DAlb}.  To obtain the second exact sequence, we apply $\Lambda_n \otimes_{\Lambda} (\cdot)$
    to the short exact sequence of \ref{thmE:str} to get a long exact sequence that fits into the commutative diagram
    \begin{equation*}
        \xymatrix{
             0 \ar[r] & {\Tor_{\Lambda}^1(\Lambda_n,\D(\T_S))} \ar[r] & {\Lambda_n \tens_{\Lambda} \D^{\mult}} \ar[r]\ar[d] & 
             {\Lambda_n \tens_{\Lambda} \D_S^{\mult}} \ar[r]\ar[d] & {\Lambda_n \tens_{\Lambda} \D(\T_S)} \ar[r]\ar[d] & 0 \\
            &  0 \ar[r] & {\D_n^{\mult}} \ar[r] & {\D_{n,S}^{\mult}} \ar[r] & {\D(\T_S)} \ar[r] & 0
        }
    \end{equation*}
    whose bottom row is \eqref{eq:DPic} and whose vertical maps are the canonical projections.  
    The rightmost vertical map is identified with the multiplication map which,
    as the $\Lambda$-module structure map of $\D(\T_S)$
    factors through the augmentation map $\Lambda\twoheadrightarrow \Lambda_n \twoheadrightarrow W$,
    is easily seen to be an isomorphism.  The middle vertical map is an isomorphism thanks to Theorem \ref{thmD} \ref{thmD:control},
    so a diagram chase yields the short exact sequence
    \begin{equation}
        \xymatrix{
            0 \ar[r] & {\Tor_{\Lambda}^1(\Lambda_n,\D(\T_S))} \ar[r] & {\Lambda_n \tens_{\Lambda} \D^{\mult}} \ar[r] &
            {\D_n^{\mult}} \ar[r] & 0
        }.\label{Tor1seq}
    \end{equation}
    On the other hand, applying $(\cdot)\otimes_{\Lambda} \D(\T_S)$ to the tautological exact sequence 
    \begin{equation*}
        \xymatrix{
            0 \ar[r] & {I_{n}} \ar[r] & \Lambda \ar[r] & \Lambda_n \ar[r] & 0
        }
    \end{equation*}
    gives the exact sequence
    \begin{equation*}
        \xymatrix{
            0 \ar[r] & {\Tor_1^{\Lambda}(\Lambda_n,\D(\T_S))} \ar[r] & 
            {I_{n} \tens_{\Lambda} \D(\T_S)} \ar[r] & {\Lambda \tens_{\Lambda} \D(\T_S)}\ar[r] & {\Lambda_n\tens_{\Lambda} \D(\T_S)}\ar[r] & 0
        }.
    \end{equation*}
    Again, as the $\Lambda$-module structure on $\D(\T_S)$ factors through $\Lambda\rightarrow \Lambda_n \rightarrow W$, 
    the final map above is an isomorphism, which yields 
    \begin{equation}
        \Tor^{\Lambda}_1(\Lambda_n,\D(\T_S))\simeq I_{n}\tens_{\Lambda}\D(\T_S)
        \simeq (I_{n} \tens_{\Lambda} W) \tens_W \D({\T_S}) \simeq \frac{I_{n}}{I \cdot I_{n}}\tens_W \D(\T_S),
        \label{Tor1desc}
    \end{equation}
    where the final isomorphism results from the canonical identification $W\simeq \Lambda/I$ via the augmentation map. 
    Putting \eqref{Tor1seq} and \ref{Tor1desc} together gives the second exact sequence in Theorem \ref{thmE} \ref{thmE:control}.
    
    It remains to establish the explicit descriptions of $\D(\T_S)$ and $\D(\Qscr_S)$ given in Theorem \ref{thmE} \ref{thmE:torus}.
    By Corollary \ref{GJstr} and the very definition of $\T_S$, we have a short exact sequence
    of $p$-divisible groups 
    \begin{equation*}
        \xymatrix{
            0 \ar[r] & {\mu_{p^{\infty}}}  \ar[r]^-{\Delta} & {\Res_{S/k} \mu_{p^{\infty}}} \ar[r] & {\T_S} \ar[r] & 0
        }
    \end{equation*}
    Applying $\D(\cdot)$ to this gives the claimed description of $\D(\T_S)$, and 
    the description of $\D(\Qscr_S)$ then follows from the very definition
    $\Qscr_S:=\T_S^{\vee}$ and the compatibility of $\D(\cdot)$ with duality.
\end{proof}

\begin{proof}[Proof of Theorem \ref{thmF}]
    Let $d$ be the dimension of the $p$-adic Lie group $\Gamma$,
    and without loss of generality assume that $\{\Gamma_n\}_n$
    is the lower central $p$-series of $\Gamma$, so that $\Gamma/\Gamma_n$
    is a finite $p$-group of order $p^{d n }$.

    The functorial decompositions 
    $\Gscr_n = \Gscr_n^{\et} \times \Gscr_n^{\mult}\times \Gscr_n^{\ll}$ yield a canonical isomorphism of $\Lambda$-modules
    $$\D\simeq \D^{\et}\times \D^{\mult}\times \D^{\ll},$$
    so as $\D^{\star}$ is finitely generated (even as a $\Lambda$-module) for $\star=\et,\mult$ thanks to Theorems \ref{thmD} and \ref{thmE},
    it suffices to prove that $\D$ is not a finitely generated $\Lambda[\![F,V]\!]$-module.
    Assume to the contrary that it is, 
    and choose $\Lambda[\![F,V]\!]$-module generators $\delta_1,\ldots,\delta_r$.
    Since the transition maps $\rho:\D_n\rightarrow \D_m$ are surjective due to Lemma \ref{lem:Dnsurj},
    the canonical projection $\D\rightarrow \D_n$ is surjective for every $n$,
    whence the images of $\delta_1,\ldots\delta_r$ generate $\D_n $
    as a $\Lambda_n[F,V]$-module. It follows that $\D_n / (F,V)$
    is generated by at most $r$ elements as a $\Lambda_n$-module, for all $n$.
    But the canonical isomorphisms $\D_n/p\D_n \simeq H^1_{dR}(X_n/k)$ of Oda's theorem induce isomorphisms
    \begin{equation*}
        D_n/(F,V) \simeq \coker\left(V : H^0(\Omega^1_{X_n})\rightarrow H^0(\Omega^1_{X_n})\right)
    \end{equation*}
    for each $n$
    so that the cokernel of $V$ is generated as an $\Omega_n$-module by at most $r$ generators.
    It follows that the cokernel---and hence the kernel---of $V$ on $H^0(\Omega^1_{X_n})$
    has $k$-dimension at most $r\cdot |\Gamma/\Gamma_n| = r p ^{dn}$.
    Writing $a_n$ for this dimension, we will prove that for {\em any} $D>0$, there exists
    $N$ with $a_n > D p^{dn}$ for all $n > N$, contradicting the above bound coming from our
    finite generation assumption.  To do so, we may---and henceforth do---assume that $k$
    is algebraically closed.
    
    Since $\Gamma/\Gamma_n$ is a $p$-group, it is solvable, and we may 
    find a subgroup $H_n \leq \Gamma/\Gamma_n$ of order $p$.
    Let $L_n\subset K_n$ be the fixed field of $H_n$, so 
    $K_n/L_n$ a degree-$p$ Galois extension, with Galois group $H_n\simeq \ZZ/p\ZZ$,
    totally ramified over the finite set of places of $L_n$ lying over $S\subseteq K_0$.
    For $Q\in S$, let $d_{Q,n}$ be the unique break in the lower numbering ramification
    filtration of $H_n$ above $Q$.  As the lower numbering of the ramification groups
    is inherited by subgroups, $d_{Q,n}$ is the largest lower break in the
    ramification filtration of $\Gamma/\Gamma_n$ above $Q$.  Working locally at $Q$, 
    let $K_{n,Q}$ be the localization of $K_n$ at $Q$, 
    and $L_Q$ the compositum of all $K_{n,Q}$. 
    Then $L_Q/K_Q$ is {\em strictly APF} \cite[Th\'eor\`eme 1.2]{Wintenberger},
    so the Herbrand functions
    $\psi(x):=\int_0^x [\Gamma: \Gamma^t] dt$ and $\phi:=\psi^{-1}$ are well-defined
    and continuos, piecewise linear increasing bijections on $[0,\infty)$.
    Let $u_n:= \phi(d_n)$ be the last upper break in the ramification filtration of $\Gamma/\Gamma_n$
    at $Q$, so $d_n = \psi(u_n)$.  By \cite[Th\'eor\`eme 1.2]{Wintenberger},
    the ratio  ${\psi(x)}/{[\Gamma: \Gamma^x]}$     
    tends to infinity with $x$. It follows that for any $D>0$,
     we have $\psi(x) > D [\Gamma: \Gamma^x]$ for all $x$ sufficiently large.
     Evaluating on $x= u_n+\varepsilon$, and noting that by definition of $u_n$
     we have $\Gamma^x \subseteq \Gamma_n$ for all $x > u_n$, we find
     that for any $D$, there exists $N$ so that, for all $n>N$ and all $\varepsilon > 0$
     \begin{equation}
         \psi(u_n+\varepsilon) > D [\Gamma: \Gamma_n] \cdot [\Gamma_n : \Gamma^{u_n + \varepsilon}]
     \end{equation}
     As $\psi$ is continuous, we conclude that $d_{Q,n}=\psi(u_n) \ge D [\Gamma: \Gamma_n]$ for all $n > N$.
     We may clearly arrange for this to hold for all $Q\in S$ simultaneously.
     
     On the other hand, by \cite[Theorem 1.1]{BooherCais} we have the lower bound
     \begin{equation*}
        a_n >  \left(1 - \frac{1}{p}\right)\left\lfloor\frac{p}{2}\right\rfloor \sum_{Q\in S}\left( 
        \left\lceil\frac{p}{2}\right\rceil \frac{d_{Q,n}}{p} - 1
        \right) 
     \end{equation*}
     and we conclude that for any $D>0$, there exists $N$ so that for all $n>N$ we have
     $a_n > D |\Gamma/\Gamma_n|$.
\end{proof}

To prove Theorem \ref{thmG}, we will need:

\begin{lem}\label{DseqAbvar}
    Let $A$ be an abelian variety over a finite field $k$ of cardinality $p^r$,
    let $\Gscr:=A[p^{\infty}]$ be the $p$-divisible group of $A$, and $\wp:=1-F^r$
    the Lang isogeny.
    There is a functorial exact sequence
    \begin{equation*}
        \xymatrix{
            0 \ar[r] & {\D(\Gscr^{\et})} \ar[r]^-{\wp} & {\D(\Gscr^{\et})} \ar[r] & {\Gscr(k)^*\otimes_{\ZZ_p} W} \ar[r] & 0
        }
    \end{equation*}
\end{lem}

\begin{proof}
    As $A$ is smooth and geometrically connected, $\wp: A\rightarrow A$ is surjective,
     and yields a short exact sequence of abelian sheaves
    \begin{equation*}
        \xymatrix{
            0 \ar[r] & {A(k)} \ar[r] & {A} \ar[r]^-{\wp} & {A} \ar[r] & 0
        },
    \end{equation*}
    from which we deduce the exact sequence 
    \begin{equation}
        \xymatrix{
            0 \ar[r] & {\Gscr(k)} \ar[r] & {\Gscr^{\et}} \ar[r]^-{\wp} & {\Gscr^{\et}} \ar[r] & 0
        }.\label{pdivseq}
    \end{equation}
    Indeed, for each $n$, multiplication by $p^n$ and the snake lemma give exact sequences
    \begin{equation}        
        \xymatrix{
            0 \ar[r] & {A[p^n](k)} \ar[r] & {A[p^n]} \ar[r]^{\wp} & {A[p^n]} \ar[r] & {A(k)/p^nA(k)} \ar[r] & 0
        }\label{pdivseqfinite}
    \end{equation}
    which, for variable $n$, form an inductive system in which the transition maps $A(k)/p^nA(k)\rightarrow A(k)/p^{n+1}A(k)$
    are mutiplication by $p$.  Passing to \'etale parts (which is exact) and applying the exact functor $\varinjlim$ to \eqref{pdivseqfinite}
    therefore yields \eqref{pdivseq}.
    Now the Dieudonn\'e module functor for 
    a constant group scheme $\underline{H}$
    is particularly simple:
    \begin{equation*}
        \D(\underline{H}) = \varinjlim_n\Hom_{k\mhyphen\mathrm{gp}}(\underline{H}, W_n) = 
    (\Hom(H, \QQ_p/\ZZ_p)\otimes_{\ZZ_p} W(\o{k}))^{\Gal(\o{k}/k)} = \Hom(H,\QQ_p/\ZZ_p)\otimes_{\ZZ_p} W
    \end{equation*}
    as the Galois action on $\underline{H}(\o{k})=H$ is trivial.
    Applying $\D(\cdot)$ to \eqref{pdivseq} therefore gives the exact sequence of Lemma \ref{DseqAbvar}.
\end{proof}

\begin{remark}
    If $G$ is {\em any} $p$-divisible group or finite $k$-group, one always has an exact sequence of \'etale abelian sheaves
    \begin{equation*}
        \xymatrix{
            0 \ar[r] & {G(k)} \ar[r] & {G^{\et}} \ar[r]^-{\wp}  & {G^{\et}}
        }
    \end{equation*}
    which yields the isomorphism of finite length $W$-modules
    $$
        \D(G^{\et})/\wp\D(G^{\et}) \simeq G(k)^*\otimes_{\ZZ_p} W.
    $$
\end{remark}

\begin{proof}[Proof of Theorem \ref{thmG}]
By Lemma \ref{DseqAbvar}, for each $n$ we have exact sequences of $\Lambda_n$-modules
\begin{equation}
    \xymatrix{
        0 \ar[r] & {\D(\Gscr_n^{\et})} \ar[r]^-{\wp} & {\D(\Gscr_n^{\et})}  \ar[r] & {\Gscr_n(k)^*\otimes_{\ZZ_p} W} \ar[r] & 0
    }\label{DmodEtLvln}
\end{equation}
that are compatibile with change in $n$ using the maps induced by Picard functoriality.
By Lemma \ref{lem:Dnsurj}, applying $\varprojlim$ results in the first short exact sequence 
of Theorem \ref{thmG} once we observe that the canonical map
\begin{equation*}
    \xymatrix{
        {M_W:=\left(\varprojlim_{n} \Gscr_n(k)^*\right) \otimes_{\ZZ_p} W} \ar[r] & 
        {\varprojlim_{n} \left(\Gscr_n(k)^* \otimes_{\ZZ_p} W\right)}
    }
\end{equation*}
is an isomorphism, due to the fact that $\Gscr_n(k)$ is {\em finite}
(hence of finite length as a $\ZZ_p$-module) and $W$ is finitely presented as a $\ZZ_p$-module.
Theorem \ref{thmE} \ref{thmE:control} and the exact sequence \eqref{DmodEtLvln}
provide the commutative diagram with exact rows and columns
\begin{equation*}
    \xymatrix{
                & & & 0 \ar[d] & \\
        0 \ar[r] & {\D(\Qscr_S)} \ar[r]\ar[d]^-{\wp} & {\Lambda_n\tens_{\Lambda} \D^{\et}} \ar[r]\ar[d]^-{1\otimes\wp} & {\D(\Gscr_n^{\et})} \ar[r]\ar[d]^-{\wp} & 0 \\
        0 \ar[r] & {\D(\Qscr_S)} \ar[r]\ar[d] & {\Lambda_n\tens_{\Lambda} \D^{\et}} \ar[r]\ar[d] & {\D(\Gscr_n^{\et})} \ar[r]\ar[d] & 0\\
            & \D(\Qscr_S)/\wp \D(\Qscr_S) \ar[d] & \coker(1\otimes \wp) \ar[d] &  {\Gscr_n(k)\tens_{\ZZ_p} W} \ar[d] & \\   
            & 0 & 0  & 0 & 
    }
\end{equation*}
Using the snake lemma, the first exact sequence of Theorem \ref{thmG}, and the fact that 
tensor product commutes with the formation of cokernels,
we deduce the exact second sequence of Theorem \ref{thmG}.
\end{proof}

\begin{remark}
    If every point of $S$ is $k$-rational, then $F^r$ acts trivially on $\D(\Qscr_S)$
    thanks to the explicit description given in Theorem \ref{thmE} \ref{thmE:torus},
    so that $\wp=0$ on $\D(\Qscr_S)$ in this case.
\end{remark}

\begin{proof}[Proof of Corollary \ref{thmH}]
    Applying $(\cdot)\otimes_W k$ to the second exact sequence of Theorem \ref{thmG} gives an exact sequence
    \begin{equation*}
        \xymatrix{
            {\Qscr_S(k)^*\tens_{\ZZ_p} k } \ar[r] & {\left(\Lambda_n \tens_{\Lambda}M_W \right)\tens_{W} k} \ar[r] & {\Gscr_n(k)^* \tens_{\ZZ_p} k } \ar[r] & 0
        }.
    \end{equation*}
    On the other hand, for each $n$ we have canonical isomorphisms 
    \begin{equation*}
        \left(\Lambda_n \tens_{\Lambda}M_W \right)\tens_{W} k \simeq \Omega_n \tens_{\Omega} \left(M_W \tens_W k\right)
    \end{equation*}
and for any abelian group $G$, we have $G^*/pG^* = (G[p])^*$, whence the exact sequence
\begin{equation*}
        \xymatrix{
            {\Qscr_S[p](k)^*\tens_{\FF_p} k } \ar[r] & {\Omega_n \tens_{\Omega}\left(M_W\tens_{W} k\right)} \ar[r] & {\Gscr_n[p](k)^* \tens_{\FF_p} k } \ar[r] & 0
        },
    \end{equation*}
    and we deduce that for all $n$,
    \begin{equation*}
        \log_p |\Gscr_n[p](k)| = \dim_k \left(\Gscr_n[p](k)\tens_{\FF_p} k \right)
        =\dim_{k} \left(\Omega_n\tens_{\Omega}(M_W \otimes_W k) \right) + O(1).
    \end{equation*}
    On the other hand, applying $(\cdot)\otimes_W k$ to the first exact sequence of Theorem \ref{thmG}
    shows that $M_W \otimes_W k$ is a finitely generated $\Omega$-module, and the
    result follows from \cite[Proposition 2.18]{EmertonPaskunas} for general $\Gamma$
    and from \cite[Theorem 1.8]{MonskyHK} when $\Gamma=\ZZ_p^d$ is abelian.
\end{proof}

\begin{remark}
    As noted in \cite[Remark 2.19]{EmertonPaskunas}, a stronger version of  
    \cite[Proposition 2.18]{EmertonPaskunas}---which would imply that we may
    {\em always} take $\mu=\nu$ in Corollary \ref{thmH}---is stated in
    \cite[Theorem 2.3]{EmertonCalegari}, which refers to 
    \cite[\S5]{AB} for a proof.
    We were unfortunately not able to 
    deduce a proof of \cite[Theorem 2.3]{EmertonCalegari} from the provided reference. 
    According to \cite[Remark 2.20]{EmertonPaskunas}, a weaker version 
    of \cite[Proposition 2.18]{EmertonPaskunas} follows from
    (the proof of) \cite[Theorem 1.10]{Harris1},
    though we have not checked the details (see also \cite{Harris2}).
    In \cite{Harris1}, Harris writes ``I understand that Y.~Ochi, in his 1998 Cambridge Ph.D.~thesis, 
    has also found [a] new proof[]~of Theorem 1.10...''.  However, in the introduction to his 1999 Cambridge Ph.D~thesis,
    Ochi writes ``...the asymptotic formlua [Theorem 1.10] of Harris will not be found in this paper'' \cite[p.~14]{OchiThesis}.
\end{remark}

\subsection{The \'etale case}

We now suppose that $\Sigma=\emptyset$; {\em i.e.}~that 
$\pi_{n,0}:X_n\rightarrow X_0$
is \'etale for all $n$,
and we recall that $S=\{x_0\}$ for a fixed closed point $x_0$ of $X_0$,
and $S_n=\pi_{n,0}^{-1}(S_0)$ is the fiber of $\pi_{n,0}$ over $S$. 
The proofs of Theorems \ref{thmA} and \ref{thmB}
proceed along much the same lines as those of Theorems \ref{thmD} and \ref{thmE},
though there are some key differences.

\begin{lem}\label{lem:DmapSurjet}
    For $\star\in \{\et,\mult,\ll\}$ and each pair of integers $n\ge m\ge 0$, the maps
    $\o{\rho}: \o{\D}_{n,S}^{\star}\rightarrow \o{\D}_{n,S}^{\star}$ are surjective.
      Moreover, $\o{\D}_{n,S}^{\star}$ is a free $\Omega_n$-module
      of rank $\gamma$ when $\star=\mult,\et$, and is free of rank $2(g-\gamma)$
      when $\star=\ll$.
\end{lem}

\begin{proof}
    As noted in Remark \ref{loclocS}, we have $\D_{n,S}^{\ll} \simeq\D_{n}^{\ll}$ for all $n$,
    so that $\D_S^{\ll}\simeq \D^{\ll}$.
    We first show that for all integers $m\le n$, the map $\o{\rho}: \o{\D}_n^{\ll}\rightarrow \o{\D}_m^{\ll}$
    is surjective, and $\o{\D}^{\ll}$ is a free $\Omega_n$-module of rank $2(g-\gamma)$.
    To do so, we may assume that $k$ is algebraically closed.
    Oda's theorem gives a commutative diagram with horizontal isomorphisms
    \begin{equation*}
        \xymatrix{
            {\o{\D}^{\ll}_n} \ar[r]^-{\simeq}\ar[d]^-{\rho} & {H^1_{\dR}(X_n/k)^{\FVnil}} \ar[d]^-{\pi_*}\\
            {\o{\D}^{\ll}_m} \ar[r]^-{\simeq} & {H^1_{\dR}(X_m/k)^{\FVnil}} 
        }
    \end{equation*}
    wherein $M^{\FVnil}$ denotes the maximal submodule of $M$ on which $F$ and $V$ are nilpotent.
    The trace map on de~Rham cohomology fits into a commutative diagram with exact rows    
     \begin{equation}
        \xymatrix{
            0 \ar[r] & {H^0(X_n,\Omega^1_{X_n/k})^{\Vnil}} \ar[r]\ar[d]^-{\pi_*} & {H^1_{\dR}(X_n/k)^{\FVnil}}   \ar[r] \ar[d]^-{\pi_*}& {H^1(X_n,\O_{X_n})^{\Fnil}} \ar[r]\ar[d]^-{\pi_*} & 0 \\
            0 \ar[r]& {H^0(X_m,\Omega^1_{X_m/k})^{\Vnil}} \ar[r] & {H^1_{\dR}(X_m/k)^{\FVnil}}   \ar[r] & {H^1(X_m,\O_{X_m})^{\Fnil}} \ar[r] & 0
        },\label{drseqsurj}
    \end{equation}
    so it suffices to prove that the flanking vertical maps are surjective
    and that each of the flanking terms (in the top row say) 
    is free over $\Omega_n$ of rank $g-\gamma$.
    As $\pi$ is separable, pullback of differential forms is injective,
    whence the dual map $\pi_*: H^1(X_n,\O_{X_n})\rightarrow H^1(X_m,\O_{X_m})$
    is surjective; as the right vertical map in \eqref{drseqsurj} is a direct summand of this map,
    it is surjective as well.  On the other hand, for {\em any} smooth and proper curve $X/k$
    with {\em reduced} divisor $S$, 
    the canonical inclusion
    $$H^0(X,\Omega^1_{X/k})\rightarrow H^0(X,\Omega^1_{X/k}(S))$$
    induces an isomorphism on $V$-nilpotent subspaces, due to the formula 
    $\res_x(V\eta)^p = \res_x(\eta)$.  Thus, the surjectivity of the left vertical map in \eqref{drseqsurj}
    follows from Proposition \ref{PicInj}.
    
    By Nakajima's Theorem \ref{Nakajima} \ref{NakNil}, the $\Omega_n$-module $H^0(X_n,\Omega^1_{X_n/k})^{\Vnil}$
    is free of rank $g-\gamma$.  Since $H^1(X,\O_X)^{\Fnil}$ is functorially dual to this by Serre duality,
    it is isomorphic to the contragredient $\Omega_n$-module and hence free of rank $g-\gamma$ as well.
    
    The argument that $\o{\D}_S^{\mult}$ and $\o{\D}_S^{\et}$ are each free $\Omega_n$-modules of rank $\gamma$
    and that the mod $p$ transition maps are surjective is essentially the same as in the proof of Lemma \ref{lem:DmapSurj}:
    by Oda's theorem and duality, it is enough to prove that $H^0(X_n,\Omega^1_{S_n}(S_n))^{\Vbij}$
    is free as a $\Omega_n$-module, and the pullback (respectively trace) map
    \begin{equation*}
        \xymatrix{
            {H^0(X_m,\Omega^1_{X_m/k}(S_m))}\ar[r]<0.75ex>^-{\pi^*} &\ar[l]<0.75ex>^-{\pi_*} {H^0(X_n,\Omega^1_{X_n/k}(S_n))}
        }
    \end{equation*}
    is {\em injective} (respectively surjective).  The desired freeness follows from Nakajima's theorem, and pullback of differentials
    along a generically \'etale (even \'etale!) map is injective.  The surjectivity of $\pi_*$
    is Corollary \ref{PicInjCor}.
\end{proof}

\begin{proof}[Proof of Theorem \ref{thmA}]
    Lemma \ref{lem:DmapSurjet} and the discussion below \eqref{Dpairing}
    show that the $\Gamma$-towers of $W$-modules $\{\D_{n,S_n}^{\mult},\rho\}$
    and $\{\D_{n,S_n}^{\et},\rho'\}$ satisfy the
    hypotheses of Propositions \ref{prop:NCfree} and \ref{prop:NCDual},
    as does $\{\D_{n,S_n}^{\ll},\rho\}$.  
    The conclusions of these Propositions immediately give Theorem \ref{thmA}.    
\end{proof}

For a finite group $G$, write $\BT_G$ for the category of $p$-divisible groups 
equipped with an action of $G$ by automorphisms; for simplicity we put $\BT:=\BT_1$.
If $H$ is a subgroup of $G$, we write $\Ind_H^G$ (respectively $\CoInd_H^G$) for the functor 
which is left (respectively right) adjoint to the forgetful functor $\BT_G\rightarrow \BT_H$.
Explicitly,
if $g_1,\ldots,g_n$ are coset representatives for $H$ in $G$, then $\Ind_H^G \Gscr$
(respectively $\CoInd_H^G$) is the co-product (resp. product) of $n$ copies of $\Gscr$, indexed by $\{g_i\}$, with obvious $G$-action.
As finite products and coproducts coincide, these functors are naturally isomorphic.
We write $\Nm: \Ind_H^G \Gscr\rightarrow \Gscr$ (respectively
$\Delta: \Gscr\rightarrow \CoInd_H^G \Gscr$) for the unique morphism corresponding to the identity map
on $\Gscr$ via the left (resp. right) adjointness of each functor.
If $H=1$, we simply write $\Ind^G$ and $\CoInd^G$ in place of $\Ind_1^G$ and $\CoInd_1^G$, respectively.
Note that $\Ind^{G}_H \circ \Ind^H = \Ind^G$,
and when $H$ is normal in $G$ and acts trivially on $\Gscr$, we 
may view $\Gscr$ in $\BT$ or in $\BT_H$, and have the formula
$\Ind^G_H \Gscr = \Ind^{G/H} \Gscr$ in $\BT_G$; analogous formulae hold for $\CoInd$.
In this way, if $H$ is normal in $G$ and acts trivially on $\Gscr$, we obtain maps in $\BT_G$
\begin{equation*}
    \xymatrix@C=50pt{
        {\CoInd^{G/H} \Gscr \simeq \CoInd_H^G \Gscr} \ar[r]^-{\CoInd_H^G(\Delta)} & {\CoInd_H^G \CoInd^H \Gscr \simeq \CoInd^G \Gscr}
        }
\end{equation*}
and
\begin{equation*}
    \xymatrix@C=50pt{
        {\Ind^{G} \Gscr \simeq \Ind_H^G \Ind^H \Gscr} \ar[r]^-{\Ind_H^G(\Nm)} & {\Ind_H^G  \Gscr \simeq \Ind^{G/H} \Gscr}
        }
\end{equation*}
that---by a slight abuse of notation---we again denote simply by $\Delta$ and $\Nm$, respectively.

\begin{lem}
    Assume that the chosen point $x_0$ is $k$-rational, and that the \'etale $k$-scheme 
    $S_n=\pi_n^{-1}x_0$ splits for all $n$.
    For each pair of positive integers $m\le n$, the map $\pi_{n,m}:X_n\rightarrow X_m$ induces commutative diagrams
    of $p$-divisible groups with exact rows
    \begin{subequations}
	\begin{equation}
    		\xymatrix@C=50pt{
        			0 \ar[r] & {\mu_{p^{\infty}}}\ar[r]^-{\Delta} & {\CoInd^{\Gamma/\Gamma_n} \mu_{p^{\infty}}} \ar[r]   & 
			{\Gscr_{n,S_n}} \ar[r]  & {\Gscr_{n}} 
			\ar[r] & 0 \\
        			0 \ar[r] & {\mu_{p^{\infty}}}\ar[r]^-{\Delta}\ar[u]^-{\id} & {\CoInd^{\Gamma/\Gamma_m} \mu_{p^{\infty}}}\ar[r]\ar[u]^-{\Delta} & {\Gscr_{m,S_m}} 
			\ar[r]	\ar[u]^-{\pi^*} & {\Gscr_{m}} 
			\ar[r]\ar[u]^-{\pi^*} & 0
    }\label{eq:Pic}
\end{equation}
and
\begin{equation}
    		\xymatrix@C=50pt{
        			0 \ar[r] & {\mu_{p^{\infty}}}\ar[r]^-{\Delta}\ar[d]_-{[\Gamma_m:\Gamma_n]} & {\Ind^{\Gamma/\Gamma_n} \mu_{p^{\infty}}}\ar[r]  \ar[d]_-{\Nm} & 
			{\Gscr_{n,S_n}} \ar[r] \ar[d]_-{\pi_*} & {\Gscr_{n}} 
			\ar[r] \ar[d]_-{\pi_*} & 0 \\
        			0 \ar[r] &{\mu_{p^{\infty}}}\ar[r]^-{\Delta} & {\Ind^{\Gamma/\Gamma_m} \mu_{p^{\infty}}} \ar[r] & {\Gscr_{m,S_m}} 
			\ar[r]	& {\Gscr_{m}} 
			\ar[r] & 0
    }\label{eq:Alb}
\end{equation}
\end{subequations}
	compatibly with change in $m,n$.
	Moreover, \eqref{eq:Pic} and \eqref{eq:Alb} are equivariant for the $(\cdot)^*$-action and the 
	$(\cdot)_*$-action of $\Gamma$, respectively.
\end{lem}

    Next, we apply $\D(\cdot)$ to the multiplicative part of \eqref{eq:Pic},
    which yields a commutative diagram with exact rows
    \begin{equation*}
        \xymatrix{
            0 \ar[r] & {\D_n^{\mult}} \ar[r]\ar[d]^-{\rho} & {\D_{n,S_n}^{\mult}} \ar[r]\ar[d]^-{\rho} & {\Lambda_n} \ar[r]\ar[d]^-{\pr} & W\ar[d]^-{\id} \ar[r] & 0 \\
            0 \ar[r] & {\D_m^{\mult}} \ar[r] & {\D_{m,S_m}^{\mult}} \ar[r] & {\Lambda_m} \ar[r] & W \ar[r] & 0 
        }
    \end{equation*}
    in which $\pr$ is the canonical quotient map,
    and we deduce a commutative diagram with short exact rows:
    \begin{equation}
        \xymatrix{
            0 \ar[r] & {\D_n^{\mult}} \ar[r]\ar[d]^-{\rho} & {\D_{n,S_n}^{\mult}} \ar[r]\ar[d]^-{\rho} & {I_{\Gamma/\Gamma_n}} \ar[r]\ar[d]^-{\pr} & 0\\
            0 \ar[r] & {\D_m^{\mult}} \ar[r] & {\D_{m,S_m}^{\mult}} \ar[r] & {I_{\Gamma/\Gamma_m}} \ar[r] & 0
        }\label{ramminv}
    \end{equation}

\begin{proof}[Proof of \ref{thmB}]
    
As in Remark \ref{loclocS}, we have $\D_{n,S_n}^{\ll}\simeq \D_{n}^{\ll}$ for all $n$
so that \ref{eta} of Theorem \ref{thmB} follows from Theorem \ref{thmA} \ref{thmA:str} and \ref{thmA:control}.
Passing to projective limits on \eqref{ramminv} and using 
Lemma \ref{lem:DmapSurjet} yields 
the second exact sequence in Theorem \ref{thmB} \ref{etb}.  To obtain the first,
we pass to multiplicative parts on \eqref{eq:Alb}, dualize, apply $\D(\cdot)$ 
and invoke the canonical identifications $\Gscr_n^{\vee}\simeq \Gscr$ 
to obtain 
\begin{equation}
    		\xymatrix@C=50pt{
        			0 \ar[r] & {W}\ar[r]^-{\Delta_n}\ar[d]^-{[\Gamma_m:\Gamma_n]} & {\Lambda_n}\ar[r]  \ar[d]^-{\pr} & 
			{\D_{n,S_n}^{\et}} \ar[r] \ar[d]^-{\rho'} & {\D_n^{\et}} 
			\ar[r] \ar[d]^-{\rho} & 0 \\
        			0 \ar[r] &{W}\ar[r]_-{\Delta_m} & {\Lambda_m} \ar[r] & {\D_{m,S_m}^{\et}} 
			\ar[r]	& {\D_m^{\et}} 
			\ar[r] & 0
    }
\end{equation}
in which $\pr$ is the canonical projection and $\Delta_n: W \rightarrow \Lambda_n$
is the ``diagonal map'' taking $\alpha\in W$ to $\alpha\sum_{\gamma\in \Gamma/\Gamma_n} \gamma$.
We break this diagram up into two diagrams with short exact rows:
\begin{equation}
    		\xymatrix@C=50pt{
        			0 \ar[r] & {W}\ar[r]^-{\Delta_n}\ar[d]^-{[\Gamma_m:\Gamma_n]} &{\Lambda_n} \ar[r]\ar[d]^-{\pr} & {\coker \Delta_n}\ar[r]  \ar[d] &  0 \\
        			0 \ar[r] &{W}\ar[r]_-{\Delta_m} & {\Lambda_m} \ar[r] & {\coker \Delta_m} \ar[r] & 0
    }\label{cokerD}
\end{equation}
and
\begin{equation}
    		\xymatrix@C=50pt{
        			0 \ar[r] & {\coker \Delta_n}\ar[d]\ar[r] & {\D_{n,S_n}^{\et}} \ar[r] \ar[d]^-{\rho'} & {\D_n^{\et}} 
			\ar[r] \ar[d]^-{\rho} & 0 \\
        			0 \ar[r] &{\coker \Delta_m} \ar[r] & {\D_{m,S_m}^{\et}} \ar[r]	& {\D_m^{\et}} 
			\ar[r] & 0
    }\label{etseqsys}
\end{equation}
As the projection map $\pr: \Lambda_n\rightarrow \Lambda_m$ is surjective for all $n\ge m$,
so too is the map
$\coker \Delta_n\rightarrow \coker \Delta_m$ thanks to the snake lemma applied to \eqref{cokerD}.
It follows that the inverse system $\coker \Delta_n$ is Mittag--Leffler, so passing to inverse limits
on \eqref{etseqsys} yields a short exact sequence of $\Lambda$-modules
\begin{equation}
    \xymatrix{
        0 \ar[r] & {\varprojlim_n \coker\Delta_n} \ar[r] & {\D^{\et}(S)} \ar[r] & {\D^{\et}} \ar[r] & 0
    }.\label{etseqalmost}
\end{equation}
On the other hand, passing to inverse limits on \eqref{cokerD} gives, 
thanks to Lemma \eqref{limitvanishing} and our assumption that $\Gamma$ is infinite, an isomorphism of $\Lambda$-modules
\begin{equation}\label{ckDel}
\Lambda=\varprojlim_{n} \Lambda_n \simeq \varprojlim_n \coker \Delta_n.
\end{equation}
Together with \eqref{etseqalmost}, this establishes the first exact sequence of Theorem \ref{thmB} \ref{etb}.

By the very construction of the first exact sequence of Theorem \ref{thmB} \ref{etb}
as the projective limit of the short exact sequence of inverse systems \eqref{etseqsys},
for each $n$ we have a commutative diagram with exact rows
\begin{equation*}
    \xymatrix{
       0\ar[r] & \Tor_1^{\Lambda}(\Lambda_n,\D^{\et})\ar[r] & {\Lambda_n\tens_{\Lambda}\left( \varprojlim_n \coker \Delta_n\right)} \ar[r]\ar@{->>}[d] &  {\Lambda_n\tens_{\Lambda}\D_{S}^{\et}} \ar[r]\ar[d]^-{\simeq} &  {\Lambda_n\tens_{\Lambda} \D^{\et}}\ar[d]\ar[r] & 0 \\
        & 0 \ar[r] & {\coker \Delta_n} \ar[r] & {\D_{n,S_n}^{\et}} \ar[r] & {\D_n^{\et}} \ar[r] & 0
    }
\end{equation*}
in which the leftmost vertical map is surjective, as was observed above, 
and the middle vertical map is an isomorphism by Theorem \ref{thmA} \ref{thmA:control}.
The snake lemma then implies that the rightmost vertical map
is an isomorphism for all $n$, and a diagram chase 
coupled with \eqref{ckDel} and \eqref{cokerD} yields the claimed isomorphism $\Tor^{\Lambda}_1(\Lambda_n,\D^{\et})\simeq W$. 
\end{proof}

\begin{remark}\label{limvanish}
    In the proofs of Theorems \ref{thmE} and \ref{thmB},
    we must show that certain derived limits  vanish.
    To do so, we have appealed to the Mittag--Leffler criterion \cite[\href{https://stacks.math.columbia.edu/tag/0598}{Tag 0598}]{stacks-project} and to Lemma \ref{limitvanishing},
    which relies on the explicit description of $\varprojlim^1$
    given in \cite[\href{https://stacks.math.columbia.edu/tag/091D}{Tag 091D}]{stacks-project}.
    In order to use these results,
    it is essential that all of our inverse systems are indexed by {\em countable} directed sets
    (see \cite[Proposition 2.3]{InvLim} and {\em cf.}~\cite[III.94 Exercise 4d]{BourbakiSet}).
    It is for this reason that
    we have assumed at the outset that the infinite pro-$p$ group $\Gamma$
    admits a {\em countable} basis of the identity consisting of open normal subgroups,
    which is equivalent to $\Gamma$ being first countable as a topological space.    
    Equivalently  
    \cite[Theorem 8.3]{HR79}, we demand that $\Gamma$ be metrizable as a topological space.
    As in Remark \ref{kcountable},
    this condition is automatic whenever $k$ itself is countable.
\end{remark}

\begin{remark}\label{imrho}
    Writing $\rho_S$ for the middle vertical map of \eqref{ramminv},
    we obtain from Theorem \ref{thmA} the description
    %
  $\ker\rho_S  = I_{\Gamma_m/\Gamma_n} \D^{\mult}_{n,S_n}$.
    Applying the snake lemma to \eqref{ramminv}, 
    and noting 
that the surjective map $\D_{n,S_n}^{\mult}\twoheadrightarrow I_{\Gamma/\Gamma_n}$ 
carries $I_{\Gamma_m/\Gamma_n}\D_{n,S_n}^{\mult}$
onto $I_{\Gamma_m/\Gamma_n}\cdot I_{\Gamma/\Gamma_n}$, we thereby deduce isomorphisms 
\begin{equation*}
       \coker(\rho: \D_n^{\mult}\rightarrow \D_m^{\mult}) 
       \simeq I_{m} / (I_{m} \cdot I+I_{n}),
\end{equation*}
%
Assuming that $\Gamma_n = P_n$ is the lower central $p$-series of $\Gamma$ as in Definition \ref{pseries},
it follows from the above and Lemma \ref{pcoker} that $\coker(\rho)$ is annhilated by $I+ p^{n-m}\Lambda$,
whence $\Im(\rho)$ contains $(I+p^{n-m}\Lambda) \D_{m}^{\mult}$. 
Thus, while $\rho$ is not surjective when $n>m$, its image is quite large.
In general, we do not know anything more about $\Im(\rho)$, 
and in particular do not know whether or not $\D^{\mult,1}$ vanishes.
\end{remark}

\begin{proof}[Proof of Corollary \ref{thmC}]
    Let $K_0:=K$ and for $n\ge 1$ inductively define $K_n$ to be
    the compositum of {\em all} unramified $\ZZ/p\ZZ$-extensions of $K_{n-1}$.
    Note that $K_{n}/K$ is a finite Galois $p$-extension, and that
    the maximal unramified $p$-extension of $K$ is the union of all $K_n$.
    Let $X_n$ be the unique smooth proper curve with function field $K_n$,
    so that $\{X_n\}_n$ is a $\Gamma$-tower of curves, with $\Gamma$
    the maximal pro-$p$ quotient of the \'etale fundamental group of $X_0$.
   
    For each $n$, we claim that the transition map $\rho:\D_{n+1}^{\mult}\rightarrow \D_n^{\mult}$
    has image contained in $p\D_n^{\mult}$.  Granting this claim, Lemma \ref{limitvanishing}
    gives $\D^{\mult}=0 = \D^{\mult,1}$ as desired.
    To prove the claim, by Oda's Theorem \ref{thm:Oda} we are reduced to
    proving that for all $n$, the trace mapping 
    \begin{equation*}
        \xymatrix{
            {\pi_* :H^0(X_{n+1},\Omega^1_{X_{n+1}})^{\Vbij}}\ar[r] & {H^0(X_n,\Omega^1_{X_n})^{\Vbij}}
            }
    \end{equation*}
    is {\em zero}.  To see this, let $\omega\in H^0(X_{n+1},\Omega^1_{X_{n+1}})^{\Vbij}$ be arbitrary, 
    and suppose that $\pi_*\omega\neq 0$.  We will derive a contradiction.  
    
    As $\pi_*\omega\neq 0$, there exists $u\in H^1(X_n,\O_{X_n})$ with $\langle \pi_*\omega,u \rangle\neq 0$,
    where
    \begin{equation*}
        \xymatrix{
            {\langle\cdot,\cdot,\rangle:  H^0(X_n,\Omega^1_{X_n/k}) \times H^1(X_n ,\O_{X_n})} \ar[r] & k
            }
    \end{equation*}
    is the canonical (perfect!) Serre duality pairing.
    We may uniquely decompose $u = u_0 + z$ with $u_0\in H^1(X_E,\O_{X_E})^{\Fbij}$
    and $z\in H^1(X_E,\O_{X_E})^{\Fnil}$.  Choose $N$ large enough so that $F^Nz = 0$;
    as $\pi_*\omega$ lies in the $V$-bijective subspace, we may find $\omega_0$
    so that $V^N \omega_0 = \pi_*\omega$.  Then
    \begin{equation*}
        \langle \pi_*\omega,u\rangle = \langle \pi_*\omega, u_0\rangle + \langle \pi_*\omega, z\rangle
        = \langle \pi_*\omega, u_0\rangle + \langle V^N\omega_0, z\rangle
        = \langle \pi_*\omega, u_0\rangle + \sigma^{-N}\langle \omega_0, F^Nz\rangle
        = \langle \pi_*\omega, u_0\rangle         
    \end{equation*}
    As $k$ is algebraically closed, the canonical map 
    \begin{equation*}
        \xymatrix{
            {H^1(X_n,\O_{X_n})^{F=1}\tens_{\FF_p} k} \ar[r] & {H^1(X_n,\O_{X_n})^{\Fbij}}
        }
    \end{equation*}
    is an isomorphism of $k$-vector spaces, so that  $H^1(X_n,\O_{X_n})^{\Fbij}$
    has a $k$-basis consisting of $F$-fixed vectors.  Since $\langle \pi_*\omega,u_0\rangle\neq 0$,
    there must exist some $F$-fixed basis vector $e$ with $\langle \pi_*\omega, e\rangle \neq 0$.
    Let $Y\rightarrow X_n$ be the unramified $\ZZ/p\ZZ$-cover of $X_n$ corresponding to $e$,
    and $E$ its function field.
    By construction, $E$ is an unramified $\ZZ/p\ZZ$-extension of $K_n$, whence a subfield of $K_{n+1}$, 
    so the covering map $\pi: X_{n+1}\rightarrow X_{n}$ factors as
    \begin{equation*}
        \xymatrix{
            \pi: X_{n+1} \ar[r]^-{\rho} & Y \ar[r]^{\tau} & X_{n}
        }.
    \end{equation*}
    We then compute
    \begin{equation*}
        \langle \pi_* \omega , e\rangle = \langle \tau_*\rho_* \omega, e\rangle = \langle \rho_* \omega, \tau^*e\rangle = 0,
    \end{equation*}
    since by construction, the pullback $\tau^*e$ of $e$ to $H^1(Y,\O_Y)^{F=1}$ is zero.
    This is a contradiction, whence $\pi_*\omega = 0$ as claimed.
    
    From  Theorem \ref{thmB} \ref{etb} we deduce a canonical isomorphism $\D_S^{\mult}\simeq I$,
    whence $I$ is a free $\Lambda$-module of rank $\gamma$ by Theorem \ref{thmA} \ref{thmA:str}.
    Writing $\Lambda':=\ZZ_p[\![\Gamma]\!]$ and $I'$ for its augmentation ideal, 
    it follows that $I'$ is a free $\Lambda'$-module of rank $\gamma$.
    One way to see this is to use the base change isomorphism
    $$
        \Lambda\tens_{\Lambda'}\Tor^1_{\Lambda'}(\FF_p,I')  \simeq \Tor^1_{\Lambda}(\Lambda\tens_{\Lambda'}\FF_p,\Lambda\tens_{\Lambda'}I')\simeq \Tor^1_{\Lambda}(k,I) =0
    $$
    together with faithful flatness of $\Lambda'\rightarrow \Lambda$ 
    to conclude that $\Tor^1_{\Lambda'}(\FF_p,I')=0$, and hence that $I'$
    is a projective module over the local ring $\Lambda'$.  The isomorphism $\Lambda\otimes_{\Lambda'}I'\simeq I$
    then shows that $I'$ has rank $\gamma$ as a $\Lambda'$-module.
    If $M$ is any finite $\Gamma$-module of $p$-power order, applying $\Hom_{\Lambda'}(\cdot,M)$ to the tautological exact sequence 
    \begin{equation*}
        \xymatrix{
            0 \ar[r] & {I'} \ar[r] & {\Lambda'} \ar[r]^-{\varepsilon} & {\ZZ_p} \ar[r] & 0
        }
    \end{equation*}
    therefore yields
    \begin{equation*}
        H^i(\Gamma,M) \simeq \Ext^i_{\Lambda'}(\ZZ_p,M) = 0
    \end{equation*}
    for $i > 1$ as well as the exact sequence 
    \begin{equation*}
        \xymatrix{
            0 \ar[r] & {\Hom_{\Lambda'}(\ZZ_p,M)} \ar[r]^{\varepsilon^*} & {\Hom_{\Lambda'}(\Lambda',M)}\ar[r] & {\Hom_{\Lambda'}(I',M)}
            \ar[r] & {\Ext^1_{\Lambda'}(\ZZ_p,M)}\ar[r] & 0
        }.
    \end{equation*}
    When $M=\FF_p$, the map $\varepsilon^*$ is an isomorphism as any $\Lambda'$-module homomorphism $\Lambda'\rightarrow \FF_p$
    factors through $\varepsilon: \Lambda'\twoheadrightarrow \ZZ_p$, and we find
    \begin{equation*}
        H^1(\Gamma,\FF_p) \simeq \Ext^1_{\Lambda'}(\ZZ_p,\FF_p) \simeq \Hom_{\Lambda'}(I',\FF_p) \simeq \FF_p^{\gamma}.
    \end{equation*}
    By Propositions 21 and 24 of \cite[Chapter I, \S4]{SerreGC}, 
     $\Gamma$ is a free pro-$p$ group on $\gamma$ generators.
\end{proof}

\begin{proof}[Proof of Corollary \ref{thmI}]
    This is proceeds along the same lines as the proof of Corollary \ref{thmH}:
    from the isomorphisms  $\D_n^{\ll}/J\D_n^{\ll}\simeq \D(\Gscr_n[J])$
    and Theorem \ref{thmB} \ref{eta} we deduce
    \begin{equation*}
        \D(\Gscr_n[J]) \simeq \Omega_n \otimes_{\Omega}\D^{\ll}/J\D^{\ll},
    \end{equation*}
    and the asymptotic formula then follows from Theorem \ref{thm:Dieudonne} \ref{lenord}, \cite[Proposition 2.18]{EmertonPaskunas},
    and---when $\Gamma=\ZZ_p^d$ is abelian---\cite[Theorem 1.8]{MonskyHK}.
\end{proof}

\begin{remark}\label{llDS}.
Let $\Gamma$ be a torsion-free $p$-adic Lie group of dimension $d$,
and $\{X_n\}$ a $\Gamma$-tower with $X_n$ corresponding to the $n$-th 
subgroup in the lower central $p$-series of $\Gamma$.
Note that when $J=(p)$, we know that 
 $\D^{\ll}/J\D^{\ll}$ is a {\em free} $\Omega$-module
 by Theorem \ref{thmB} \ref{eta}, and it follows from
  \cite[Proposition 3.5 (ii)]{Venjakob}
  that $\D^{\ll}/p\D^{\ll}$ has dimension $\delta=d$.
  The asymptotic formula of Corollary \ref{thmH} therefore yields
  constants $\nu \ge \mu \ge \frac{1}{d!}$ with
  $\mu p^{nd} + O(p^{(d-1)n})\le\log_p|\Gscr_n^{\ll}[p]|\le\nu p^{nd} + O(p^{(d-1)n})$. 
  On the other hand, by definition
$\log_p|\Gscr_n^{\ll}[p]| = h_n$,
where $h_n$ is the height of the $p$-divisible group $\Gscr_n$, and---since $\Gscr_n$
is the $p$-divisible group of the Jacobian of $X_n$---one {\em knows} that $h_n=2(g_n-\gamma_n)$,
where $g_n$ and $\gamma_n$ are the Genus and $p$-rank of $X_n$, respectively.
From the Riemann--Hurwitz and Deuring Shafarevich formulae we thus obtain the {\em exact} formula
\begin{equation*}
    \log_p|\Gscr_n^{\ll}[p]| = 2(g_0-\gamma_0)\cdot p^{dn}
\end{equation*}
so we may take $\mu=\nu=2(g_0-\gamma_0)$ in this case.  One has analogous exact formulae when $J=(F)$ or $J=(V)$;
in general, the order of $\Gscr_n^{\ll}[J]$ can be {\em extremely} subtle,
and we do not know of any exact formula beyond the three cases $J=(p)$, $J=(F)$, $J=(V)$ we have mentioned.
In this way, we may view Corollary \ref{thmI} as an asymptotic generalization of
the Deuring--Shafarevich formula to local--local $p$-torsion subgroup schemes in unramified $\Gamma$-towers.
\end{remark}

\bibliographystyle{amsalpha}
\bibliography{iw}

\end{document}